%% file: main.tex
\documentclass[onefignum,onetabnum]{siamonline190516}

\usepackage{amsmath, amsfonts, amsopn, xcolor}
\usepackage{hyperref, cleveref}
\usepackage{graphicx, epstopdf,subfigure}
\usepackage{caption, booktabs, float, multirow}
\usepackage{algorithm, algorithmic}
\usepackage{enumerate}
\usepackage{MnSymbol}
\usepackage{lipsum}
\usepackage{bm}

\ifpdf
  \DeclareGraphicsExtensions{.eps,.pdf,.png,.jpg}
\else
  \DeclareGraphicsExtensions{.eps}
\fi

\newsiamremark{remark}{Remark}
\newsiamremark{hypothesis}{Hypothesis}
\crefname{hypothesis}{Hypothesis}{Hypotheses}
\newsiamthm{claim}{Claim}

\headers{Optimization using the Koopman Operator}{Mengqi Hu, Bian Li, Yi-An Ma, Yifei Lou, and Xiu Yang}

\title{A Gradient-Based Optimization Method Using the Koopman Operator
\thanks{Submitted to the editors December 20, 2023.
\funding{
XY was supported by National Science Foundation (NSF) CAREER DMS-2143915. YL was partially supported by NSF CAREER  DMS-1846690. YM was supported in part by the National Science Foundation Grants NSF-SCALE MoDL(2134209) and NSF-CCF-2112665.
}}}

\author{
Mengqi Hu\thanks{Department of Mathematics \& School of Data Science and Society, University of North Carolina at Chapel Hill, NC
  (\email{humengqi@unc.edu}, \email{yflou@unc.edu})}
\and
Bian Li\thanks{Department of Industrial and Systems Engineering, Lehigh University, Bethlehem, PA 
  (\email{bil215@lehigh.edu}, \email{xiy518@lehigh.edu})}
\and 
Yi-An Ma\thanks{Halicio\u{g}lu Data Science Institute \& Department of Computer Science and Engineering, University of California San Diego, La Jolla, CA (\email{yianma@ucsd.edu})}
\and
Yifei Lou\footnotemark[2]
\and
Xiu Yang\footnotemark[1]
}

\newcommand{\dif}{\mathrm{d}}
\newcommand{\tensor}[1]{\mathbf{#1}}
\newcommand{\comm}[1]{} 

\newcommand{\vphin}{\varphi^N}
\newcommand{\bvphin}{\bm{{\varphi}}^N}

\newcommand{\mr}{\mathbb{R}}

\newcommand{\vphi}{\varphi}

 \newcommand{\h}[1]{\bm{#1}}

\newcommand{\trans}{^\mathsf{T}}

\ifpdf
\hypersetup{
  pdftitle ={A Gradient-Based Optimization Method Using the Koopman Operator},
  pdfauthor={Mengqi, Hu, Bian Li, Yifei Lou, and Xiu Yang}
}
\fi

\begin{document}

\maketitle

\begin{abstract}
In this paper, we propose a novel approach to solving optimization problems by reformulating the optimization problem into a dynamical system, followed by the adaptive spectral Koopman (ASK) method. The Koopman operator, employed in our approach, approximates the evolution of an ordinary differential equation (ODE) using a finite number of eigenfunctions and eigenvalues. We begin by providing a brief overview of the Koopman operator and the ASK method. Subsequently, we adapt the ASK method for solving a general optimization problem. Moreover, we provide an error bound to aid in understanding the performance of the proposed approach, marking the initial step in a more comprehensive numerical analysis. Experimentally, we demonstrate the applicability and accuracy of our method across a diverse range of optimization problems, including min-max problems. Our approach consistently yields smaller gradient norms and higher success rates in finding critical points compared to state-of-the-art gradient-based methods. We also observe the proposed method works particularly well when the dynamical properties of the system can be effectively modeled by the system's behaviors in a neighborhood of critical points.
\end{abstract}









\section{Introduction}
Optimization plays a fundamental role in many areas of science, engineering, and machine learning.
One of the most fundamental optimization methods is based on the gradient of an objective function   since the gradient at a given point provides information about the direction of the steepest increase of the function at that point. By moving along the opposite direction of the gradient, one can iteratively minimize the function, while navigating the search space until a certain type of optimal solution is reached.  This technique is particularly effective when dealing with continuously differentiable functions.
There are four major categories of gradient-based minimization algorithms, each with its own characteristics and advantages: 
\begin{itemize}
    \item \textit{Gradient Descent} (GD). This is the most basic form of gradient-based optimization. It involves iteratively
taking a small step along the opposite direction of the gradient at the current point. The step size is determined by a hyperparameter, which is referred to as a learning rate in deep learning. 
    \item \textit{Gradient Descent with Momentum} (GDM). This technique accelerates the standard GD by incorporating a momentum term when advancing to a new iterate. This addition helps to avoid local minima and often leads to faster convergence \cite{POLYAK19641}.
    \item \textit{Stochastic Gradient Descent} (SGD). As a variant of gradient descent, SGD computes the gradient using a random subset of the data at each iteration. The additional randomness enhances the algorithm's robustness to initial conditions and noisy data, facilitating convergence to a critical point \cite{bottou2010large,bottou2012stochastic,ruder2016overview}. 
    \item \textit{Adaptive Gradient Methods} (AGM). To deal with noisy data and non-convex objective functions, AGM chooses the learning rate for each parameter individually based on their historical gradients.   
    It has been demonstrated that this type of algorithm, including Adaptive subgradient methods (AdaGrad) \cite{duchi2011adaptive}
    and Adaptive Moment Estimation (Adam) \cite{kingma2014adam},
is particularly effective for problems with sparse gradients or ill-conditioned objective functions.
\end{itemize}


As we explore the realm of optimization using gradient-based methods, we are also intrigued by a specific type of optimization problem that involves the simultaneous minimization of one set of variables and the maximization of another set. This type of optimization is commonly referred to as a min-max problem, however,
the duality adds a layer of complexity, akin to fundamental concepts found in game theory. The von Neumann minimax Theorem, a cornerstone in game theory and optimization \cite{v1928theorie}, forms a bridge between min-max optimization problems and zero-sum two-player games. Specifically with finite strategy spaces, this theorem asserts that in a zero-sum two-player game with perfect information, both players' optimal strategies can be determined by solving a min-max optimization problem. This connection highlights the intriguing interactions between optimization and game theory in the pursuit of efficient solutions.

In this work, we develop an innovative gradient-based technique for optimization problems using the adaptive spectral Koopman (ASK) method~\cite{li2023adaptive}, originally designed for solving ordinary differential equations (ODEs).
The ASK method relies on the Koopman operator~\cite{koopman1931hamiltonian,koopman1932dynamical}, an infinite-dimensional linear operator for capturing the intricate nonlinearity in dynamical systems described by differential equations. Our approach shifts the perspective on optimization by interpreting it as a dynamical system rather than a standalone problem. This reinterpretation allows us to seek an equilibrium of the system through a finite-dimensional approximation achieved by combining various techniques, including eigen-decomposition, polynomial interpolation, and the Koopman operator.
In contrast to traditional optimization methods that typically depend on fixed or adaptively chosen step sizes,  the ASK method continuously evolves the system, bringing in the adaptability of the optimization process. This transition can be understood as a shift from time-domain discretization, as found in traditional gradient methods, to spatial-domain discretization. 
This innovative interpretation of gradient-based algorithms, elucidated by the ASK method, opens the door to potential improvements in the performance and convergence properties of optimization techniques. 
The main contributions of this paper are threefold.
\begin{enumerate}
    \item[(1)] We propose an efficient workflow using the Koopman operator to obtain an equilibrium of a dynamical system, which coincides with a critical point of the corresponding optimization problem. 
    \item[(2)] By further integrating the idea of sparse grids  \cite{li2022sparse} into our method, we significantly reduce the computational costs of the proposed method. Additionally, the implementation of an adaptive scheme allows our approach to handle complex problems with increased efficiency. We further analyze an error bound for a special case. 
    \item[(3)] We conduct extensive experiments on standard testing functions in both minimization and min-max problems, demonstrating that the proposed algorithm achieves significant improvements in accuracy and success rates compared with some existing gradient-based optimization techniques. 
\end{enumerate}


The rest of the paper is organized as follows. Section \ref{sect:literature} provides a concise review of the relevant literature in optimization, the Koopman operator, and the ASK method. In Section \ref{sect:proposed} we describe the proposed approach with a particular focus on the modifications made to the ASK method. Section \ref{sect:experiment} is devoted to numerous examples including three bowl-shaped testing functions, four valley-shared functions, two high-dimensional functions, and four min-max problems to showcase the efficiency of the proposed algorithm in comparison to various baseline algorithms. Finally, Section \ref{sec:dicussion_conclusion} concludes the paper.

\section{Literature Review}\label{sect:literature}

We present the formulations of a general minimization and a min-max problem in \Cref{sect:review-opt}, along with gradient-based optimization algorithms. In \Cref{sect:review-Koopman}, we provide a brief review of the concepts of the Koopman operator. \Cref{sect:ASKreview} is devoted to its applications in solving an ODE via an adaptive spectral method.

\subsection{Optimization}\label{sect:review-opt}
Consider a general unconstrained 
minimization problem, defined by
\begin{equation}
    \min_{\h x\in\mathbb R^{d}}  f(\h x),
    \label{eq:setup}
\end{equation}
where $f: \mathbb R^d\rightarrow \mathbb R$ is a continuously differentiable function to be optimized, preferably a smooth  $C^{\infty}$ function.
Traditionally,  gradient-based methods are employed to solve \eqref{eq:setup} by gradually decreasing the objective function $f$ following the negative gradient direction. 
The most straightforward approach to solving the problem as a dynamical system is through a generic formula of gradient descent (GD) algorithms, which aims to minimize $f(\h{x})$. This formula can be expressed as
\begin{align}
    \begin{cases}
        \h p^{(k+1)} &= -\nabla f(\h x^{(k)})\\
        \h x^{(k+1)} &= \h x^{(k)} + \alpha^{(k+1)}\h p^{(k+1)},
    \end{cases}\label{eq:gradient descend}
\end{align}
where $k$ denotes the iteration number and $\alpha^{(k+1)}>0$ represents a positive step size that can either be fixed or updated iteratively. 
Numerous variants of gradient descent exist, differentiated by how the step size is determined. For instance, steepest descent (SD) involves an exact line search to achieve the maximum descent along the gradient direction, making the descent the steepest. Its procedure can be outlined as follows,
\begin{align}
    \begin{cases}
        \h p^{(k+1)} &= -\nabla f(\h x^{(k)})\\
        \alpha^{(k+1)} &= \arg\min_{\alpha} f(\h x^{(k)} + \alpha\h p^{(k+1)})\\
        \h x^{(k+1)} &= \h x^{(k)} + \alpha^{(k+1)}\h p^{(k+1)}.
    \end{cases}\label{eq:steepest descend}
\end{align}
However,  SD does not work well empirically in most cases, since such a local descending property does not necessarily coincide with the overall descending of the original function.
Noting that the search direction in each iteration of the scheme \eqref{eq:steepest descend} solely relies on the information at the current step $\h{x}^{(k)}$, one can also incorporate previous steps into the iteration, which gives rise to momentum-based algorithms \cite{boyd2004convex,hu2023accelerated}. The term ``momentum" draws an analogy to a massive ball rolling on the surface of the objective function, where the update of each step is retained and utilized during the process.
Consequently, the following iteration 
\begin{align}
    \begin{cases}
    \h p^{(k+1)} &= -\nabla f(\h x^{(k)}) + \beta^{(k+1)}\h p^{(k)} \\
    \h x^{(k+1)} &= \h x^{(k)} + \alpha^{(k+1)}\h p^{(k+1)},
    \end{cases}\label{eq:grad momentum}
\end{align}
is referred to as a heavy ball (HB) algorithm \cite{POLYAK19641}. Both $\alpha^{(k+1)}$ and $\beta^{(k+1)}$ in Equation \eqref{eq:grad momentum} can be either fixed or adaptively selected based on a specific scheme.
A distinct category of momentum-based algorithms was developed by Yurii Nesterov \cite{nesterov1983method,nesterov2003introductory}. Beginning with $t^{(0)} = 1$, Nesterov's accelerated gradient (NAG) is formulated by,
\begin{align}
    \begin{cases}
    t^{(k+1)} &= \frac{ 1+\sqrt{4(t^{(k)})^2 + 1} }{2}\\
    \h p^{(k+1)} &= -\nabla f\left( \h x^{(k)} \right)\\
    \h y^{(k+1)} &= \h x^{(k)} + \alpha^{(k+1)}\h p^{(k+1)}\\
    \h x^{(k+1)} &= \h y^{(k+1)} + \frac{t^{(k)} - 1}{t^{(k+1)}}(\h y^{(k+1)} - \h y^{(k)}).
    \end{cases}\label{eq:Nest accel grad}
\end{align}
Similar to other gradient-based algorithms, the step size $\alpha^{(k+1)}$ in NAG can be constant or updated during the iteration. For any convex objective function $f(\cdot)$, NAG achieves a convergence rate of $O(\frac{1}{k^2})$, which is an improvement over the rate of $O(\frac{1}{k})$ obtained by standard gradient-based methods. This momentum scheme can be further accelerated by employing a suitable restart strategy with provable guarantees under certain conditions \cite{nemirovski1985optimal,giselsson2014monotonicity,su2014differential}.




We are also interested in a min-max problem formulated as follows,
\begin{equation}\label{eq:minmax setup}
    \min_{\h x\in \mathbb{R}^{m}}\max_{\h y\in \mathbb{R}^{n}}  f(\h x,\h y),
\end{equation}
for a continuously differentiable function $f(\cdot, \cdot)$.
Solving \eqref{eq:minmax setup} by gradient-based methods
amounts to iterating between gradient descend on $\h x$ and gradient ascend on $\h y$. In particular, a gradient descend/ascend (GDA) algorithm~\cite{ruder2016overview}  is expressed as follows,
\begin{align}
\begin{cases}
        \h x^{(k+1)} &= \h x^{(k)} - \alpha^{(k+1)} \nabla_{\h x} f(\h x^{(k)},\h y^{(k)})\\
        \h y^{(k+1)} &= \h y^{(k)} + \alpha^{(k+1)} \nabla_{\h y} f(\h x^{(k)},\h y^{(k)}),
\end{cases}
\label{eq:gradient descend ascend}
\end{align}
where  $\alpha^{(k+1)}>0$ represents a positive step size that can either be fixed or updated iteratively. 
Unfortunately, there exist instances where the system of equations in \eqref{eq:gradient descend ascend} exhibits a cyclic behavior. Motivated by a specific function $f$ with an initial point  that leads to such cycles, Daskalakis et al.~\cite{daskalakis2017training} proposed an optimistic gradient descent/ascent (OGDA) approach,
\begin{align}
\begin{cases}    
        \h x^{(k+1)} &= \h x^{(k)} - \alpha^{(k+1)} \nabla_{\h x} f(\h x^{(k)},\h y^{(k)}) -  \alpha^{(k+1)}( \nabla_{\h x} f(\h x^{(k)},\h y^{(k)}) -  \nabla_{\h x} f(\h x^{(k-1)},\h y^{(k-1)})) \\
        \h y^{(k+1)} &= \h y^{(k)} + \alpha^{(k+1)} \nabla_{\h y} f(\h x^{(k)},\h y^{(k)}) + \alpha^{(k+1)} (\nabla_{\h y} f(\h x^{(k)},\h y^{(k)})-\nabla_{\h y} f(\h x^{(k-1)},\h y^{(k-1)})),
\end{cases}
    \label{eq:Optimistic Gradient Descent Ascent}
\end{align}
and established its convergence for bilinear objective functions, i.e., $f(\h x, \h y) = \h x^{\trans} A \h y$ \cite{daskalakis2017training}. 
It is worth noting that OGDA resembles the HB method  \cite{POLYAK19641,boyd2004convex},
but a key distinction  lies in the presence of the ``negative momentum'' in OGDA, while the HB method employs a ``positive momentum.'' In addition to gradient-based methods, alternative methods for solving the min-max problem \eqref{eq:minmax setup} are documented in \cite{boyd2004convex, ben2009robust, laraki2012semidefinite, raghunathan2018semidefinite}.
All the aforementioned algorithms  can be understood in the framework of solving a certain ODE defined by, 
\begin{align}
    \begin{cases}
    \mathcal{D}\h x(t) = \h u(\h  x)\\
    \h x(t_0) = \h x_{0},
    \end{cases}\label{eq:Continuous Grad Flow}
\end{align}
where $\mathcal{D}$ is a linear differential operator, $\h u $ is a continuous function which is usually chosen as $\h u(\h x) = -\nabla f(\h x)$ for minimization and $\h u(\h x, \h y) = -\nabla f(\h x, -\h y)$ for min-max problems. It is straightforward that the gradient descend method \eqref{eq:gradient descend} can be formulated by the forward Euler scheme.  As for the momentum acceleration methods, the HB   formula \eqref{eq:grad momentum} applied the Euler's method to the following ODE  \cite{JMLR:v22:20-195}, 
\begin{align}
    \ddot{\h x} + \epsilon\dot{\h x} +\nabla f(\h x) = \h 0,
    \label{eq:polyak ODE}
\end{align}
where $\epsilon$ can be viewed as a viscosity coefficient in a damping system. 
From an optimization perspective,  the ideal value for $\epsilon$ is associated with the Lipschitz and strong convexity properties of the function $f$ \cite{siegel2019accelerated}.
Additionally, NAG is related to the ODE~\cite{su2014differential}
\begin{align}    
    \ddot{\h x} + \frac{3}{t}\dot{\h x}+\nabla f(\h x) = \h 0.
    \label{eq:nesterov ODE}
\end{align}
Note that it is necessary for $\nabla f(\mathbf{x})$ to be Lipschitz continuous with a Lipschitz constant $L$, and the step size should satisfy $\alpha < \frac{1}{cL}$ for some constant $c$ to ensure the guaranteed acceleration achieved by these methods. This requirement implies that the step size $\alpha$ should be carefully tuned, and its optimal value may vary from one problem to another \cite{su2014differential}. There is a wealth of research on continuous gradient flows and their acceleration schemes,  such as \eqref{eq:polyak ODE} \eqref{eq:nesterov ODE}, using higher-order ODEs \cite{santambrogio2017euclidean,wang2022accelerated}. However, most of them still adopt traditional numerical ODE methods that require careful tuning of the step size.


\subsection{Koopman operator}\label{sect:review-Koopman}
We review the Koopman operator \cite{koopman1932dynamical, mezic2005spectral}, which maps a finite-dimensional nonlinear system to a linear one. This mapping enables us to utilize the solution of an infinite-dimensional linear system to find an equilibrium of the nonlinear system. Define a family of the Koopman operators by
\begin{align}\label{eq:koopman}
    \mathcal{K}_{t} g(\h x(t_0)) = g\Big(\h x(t_0+t)\Big),
\end{align}
where $g$ is any (scalar) observable function of the state variables $\h x(t)$  at time $t$.  The ``derivative'' of $\mathcal K_t$, referred to as the \textit{infinitesimal generator} (or simply generator) of the Koopman operators, is defined by
\begin{align}
    \mathcal{K} g = \lim_{t\to 0}\frac{\mathcal{K}_{t}g - g}{t}\label{eq:generator}.
\end{align}
As analogous to the chain rule in calculus, one has
\begin{align}\label{eq:generator2}
    \mathcal{K} g(\h x) = \nabla g(\h x) \cdot \frac{\dif\h x}{\dif t} = \frac{\dif g(\h x)}{\dif t}.
\end{align}






Suppose $\varphi$ is an eigenfunction of $\mathcal K$ with the associated eigenvalue $\lambda$, i.e.,  $\mathcal K \varphi (\h x) = \lambda \varphi (\h x)$. It follows from  \eqref{eq:generator2} that $\lambda \varphi (\h x) = \mathcal K\varphi(\h x) = \frac{d\varphi(\h x)}{dt}$, which implies that $\varphi (\h x(t_0 + t)) =\exp(\lambda t) \varphi (\h x(t_0))$. It follows from the definition of the Koopman operator \eqref{eq:koopman}   that
\begin{align}\label{eq:eig-Kt}
    \mathcal K_{t} \varphi(\h x(t_0))=\varphi (\h x(t_0 + t)) = \exp(\lambda t) \varphi (\h x(t_0)).
\end{align}

If $g$ belongs to the function space spanned by all the eigenfunctions $\varphi_{j}$ of $\mathcal K_t$, i.e., $g(\h x) = \sum_{j}^{\infty} c_{j} \varphi_{j}(\h x)$ with coefficients $\{c_j\}$,  the linearity of $\mathcal K_t$ gives
\begin{align}
    \mathcal{K}_{t}[g(\h x(t_{0}))] = \mathcal{K}_{t}[\sum_{j}^{\infty} c_{j} \varphi_{j}(\h x(t_0))] = \sum_{j}^{\infty} c_{j} \mathcal{K}_{t}[\varphi_{j}(\h x(t_0))].
\end{align}
The above equation implies that
\begin{align}
    g(\h x(t_0 + t)) = \sum_{j}^{\infty} c_{j} \varphi_{j}(\h x(t_{0})) \exp(\lambda_{j}t),\label{eq:g}
\end{align}
where $\{(\varphi_j,\lambda_j)\}_j^\infty$ are eigenpairs of the Koopman operator $\mathcal K_t$, each satisfying \eqref{eq:eig-Kt}.

\subsection{Adaptive Spectral Koopman (ASK) methods}\label{sect:ASKreview}

Li et al.~\cite{li2023adaptive} proposed a truncated version of \cref{eq:g} for solving the dynamical system governed by $\dot{\h x}( t) = \h u(\h x),$ where 
$\h u\in\mathbb R^d$ is often called the dynamics of the state variable  $\h x\in\mathbb R^d$. 
Given an integer $N$, suppose each eigenfunction $\varphi_{j}$ of $\mathcal K_t$ can be approximated by interpolation polynomials, denoted as $\varphi_{j}^{N}$. Then, a finite-dimensional approximation of $g$ in \eqref{eq:g} can be expressed by
\begin{align}    \label{eqn:trun_expansion}
    g(\h x(t_0 + t)) \approx g^N(\h x(t_0+t)) := \sum_{j=0}^{N} \tilde{c}_{j} \varphi^{N}_{j}(\h x(t_{0})) \exp (\tilde\lambda_{ j } t ),
\end{align}
where $\lambda_{j}, c_{j}$ are approximated by $\tilde{\lambda}_{j}$ and $\tilde{c}_{j}$, respectively. For each eigenfunction $\varphi^{N}_{j},$ Li et al.~\cite{li2022sparse} further considered a finite interpolation 
on a set of (sparse) grid points. 
Specifically, 
a finite approximation of an arbitrary function $\vphi$ is given by
\begin{align}\label{eq:poly-approx}
    \vphi(\h x) \approx \vphi^N(\h x) = \sum_{l=1}^N w_l \Psi_l(\h x),
\end{align}
where $\{\Psi_l\}$ are Chebyshev polynomials. When evaluating $\h x$ on a set of sparse grid points, denoted by $\{ \bm \xi_l \}_{l=1}^N$ with $\bm \xi_l := (\xi_{l_1}, \xi_{l_2}, \dotsc, \xi_{l_d}) \in \mr^d$, \eqref{eq:poly-approx} can be written by a matrix-vector multiplication, i.e.,
\begin{align}
    \begin{bmatrix}
        \vphi^N(\bm \xi_1) \\
        \vphi^N(\bm \xi_2) \\
        \vdots \\
        \vphi^N(\bm \xi_N)
    \end{bmatrix}
    = 
    \begin{bmatrix}
        \Psi_1(\bm \xi_1) & \Psi_2(\bm \xi_1) & \dots & \Psi_N(\bm \xi_1) \\
        \Psi_1(\bm \xi_2) & \Psi_2(\bm \xi_2) & \dots & \Psi_N(\bm \xi_2) \\
        \vdots & \vdots & \ddots & \vdots \\
        \Psi_1(\bm \xi_N) & \Psi_2(\bm \xi_N) & \dots & \Psi_N(\bm \xi_N) \\
    \end{bmatrix}
    \begin{bmatrix}
        w_1 \\
        w_2 \\
        \vdots \\
        w_N 
    \end{bmatrix},\label{eqn:sparse_grid_interp}
\end{align}
or $\bm\vphi^N=\tensor M\h w$ for short, where each element of $\tensor M$ is given by $M_{ij}=\Psi_j(\bm\xi_i)$ and $\h w = (w_1, \cdots, w_N)\trans.$ Note that $\vphi$ denotes a function, while the bold-faced version $\bm\vphi$ denotes the function evaluated at a set of points. 

Taking partial derivatives of \eqref{eq:poly-approx} with respect to $x_i$ yields
\begin{align}\label{eq:finite-approx}
\frac{\partial \vphi^N(\bm x)}{\partial x_i} = \sum_{l=1}^N w_l \frac{\partial \Psi_l(\bm x)}{\partial x_i} .
\end{align}
Denote the matrix $\tensor G_i$ as $\partial_{x_i}\Psi_l(\bm x)$
evaluated at $\{ \bm \xi_l \}_{l=1}^N$, i.e.,
\begin{align}\label{eqn:Gi}
    \tensor G_i = 
    \begin{bmatrix}
         \frac{\partial \Psi_1}{\partial x_i} (\bm \xi_1) & \frac{\partial \Psi_2}{\partial x_i} (\bm \xi_1) & \dots & \frac{\partial \Psi_N}{\partial x_i} (\bm \xi_1) \\
        \frac{\partial \Psi_1}{\partial x_i} (\bm \xi_2) & \frac{\partial \Psi_2}{\partial x_i} (\bm \xi_2) & \dots & \frac{\partial \Psi_N}{\partial x_i} (\bm \xi_2) \\
        \vdots & \vdots & \ddots & \vdots \\
        \frac{\partial \Psi_1}{\partial x_i} (\bm \xi_N) & \frac{\partial \Psi_2}{\partial x_i} (\bm \xi_N) & \dots & \frac{\partial \Psi_N}{\partial x_i} (\bm \xi_N) \\
    \end{bmatrix}\in \mr^{N \times N},
\end{align}
then  \eqref{eq:finite-approx} can be expressed by
\begin{equation}\label{eq:finite-approx-matrix-form}
   \frac{\partial \vphi^N(\bm x)}{\partial x_i} = \tensor G_i \bm w.
\end{equation}

Now we regard each component of the dynamic $\h u$ as an observable that the Koopman operator operates on, i.e., $\h u:=  (g_1(\h x), g_2(\h x), . . . , g_d(\h x))\trans$, then a system of observables becomes
\begin{align}\label{eq:sys-obs}
    \frac{\dif\h u(\h x)}{\dif t} = \mathcal{K} \h u(\h x) = 
    \begin{bmatrix}
        \mathcal{K} g_{1}(\h x)\\
        \mathcal{K} g_{2}(\h x)\\
        \cdots\\
        \mathcal{K} g_{d}(\h x)
    \end{bmatrix}
    = \sum_{j}^{\infty} \lambda_{j} \varphi_{j}(\h x)\h c_{j},
\end{align}
where $\h c_{j} = (c_j^{1},\cdots, c^{d}_{j})\trans \in \mathbb{C}^{d}$ is called the $j$-th Koopman mode of the system.
It follows from \eqref{eq:generator2} that 
$\mathcal{K} \vphi(\h x) = \frac{\dif \h x}{\dif t} \cdot \nabla \vphi(\h x)$ for any function $\varphi$. 
Since $\frac{\dif \h x}{\dif t} = \h u(\h x)$, we have
\begin{align}
    \mathcal{K} \vphi = \h u \cdot \nabla \vphi = g_1 \frac{\partial  \vphi}{\partial x_1} + g_2\frac{\partial \vphi}{\partial x_2} + \dotsc + g_d\frac{\partial \vphi}{\partial x_d}.
    \label{eqn:generator_expansion}
\end{align}
It follows from \eqref{eq:finite-approx-matrix-form} that the Koopman operator $\mathcal{K}$ given in \eqref{eqn:generator_expansion} has a finite approximation by $\tensor K \in \mr^{N \times N}$, that is,
\begin{align}
    \label{eqn:ASK_finite_approximation}
    \tensor K \vphi^N = \sum_{i=1}^d \: \text{diag}{\big( g_i(\bm \xi_1), \dotsc, g_i(\bm \xi_N) \big)} \: \tensor G_i \bm w,
\end{align}
which holds for the function $\vphi^N$ as an approximation to the function $\vphi$ as in \eqref{eq:poly-approx}. The choice of  $\vphi^N$ determines the coefficient vector $\h w.$

Suppose the discretized Koopman operator $\tensor K$  has an eigenpair $(\vphi^N_j, \tilde \lambda_j)$, i.e., $\tensor K \vphi^N_j=\tilde \lambda_j\vphi^N_j.$ 
 Applying \cref{eqn:ASK_finite_approximation} on  $\vphi^N_j$ yields
\begin{align}
    \sum_i^d \: \text{diag}{\big( g_i(\bm \xi_1), \dotsc, g_i(\bm \xi_N) \big)} \: \tensor G_i \bm w_j =\tensor K \vphi^N_j= \tilde \lambda_j \tensor M \bm w_j,
    \label{eqn:ASK generalized_eigen_problem}
\end{align}
or shortly $\tensor U \bm w_j = \tilde \lambda_j \tensor M \bm w_j$ with 
$\tensor U := \sum_i^d \: \text{diag}{\big( g_i(\bm \xi_1), \dotsc, g_i(\bm \xi_N) \big)} \: \tensor G_i$. Given $\tensor U$ and $\tensor M,$ the coefficients $\h w_j$ can be solved by a generalized eigenvalue problem. For compactness, we write it in a matrix form 
\begin{equation}
  \label{eqn:generalized_eig}
\tensor U \tensor W = \tensor M \tensor W \tensor \Lambda, 
\end{equation}
where
$ \tensor W := 
    \begin{bmatrix}
        \bm w_1 & \bm w_2 & \dots & \bm w_N
    \end{bmatrix}$ and $
    \tensor \Lambda := \text{diag}\left(\tilde \lambda_1, \tilde \lambda_2,
    \dotsc, \tilde\lambda_N\right)$.
Then, the eigenfunction matrix can be defined by $\tensor \Phi^N := \tensor M
\tensor W$, whose $j$th column is $\bvphin_j$. 


Setting $t = 0$ in \eqref{eqn:trun_expansion} with a particular component $g_i$ in the dynamics $\h u$, we can compute  Koopman modes $\tilde c_{ji}$ from 
\begin{align*}
     g_i(\bm x_0) \approx \sum_{j = 1}^N \tilde c_{ji} \vphi_j^N(\bm x_0) =: g_i^N(\bm x_0),
\end{align*}
which must be satisfied for different initial conditions in the neighborhood of $\bm x_0$.
Specifically for all sparse grid points, we have 
\begin{align}\label{eqn:g-temp}
    g_i^N(\bm \xi_l) = \sum_{j = 1}^N \tilde c_{ji} \vphi_j^N(\bm \xi_l), \quad l = 1, 2, \dotsc, N,
\end{align}
which can be expressed  in a matrix form $\tensor \Phi^N \tilde{\bm c}_i = \tensor \Xi_i$ with
\begin{align}
    \tensor\Xi_i := 
    \begin{bmatrix}
        g_i(\bm \xi_1) \\
        \vdots\\
        g_i(\bm \xi_l) \\
        \vdots \\
        g_i(\bm \xi_N)
    \end{bmatrix} \quad\mbox{and}\quad
   \tilde{\h c}_i = \begin{bmatrix}
      \tilde c_{1i}\\
       \vdots\\
      \tilde c_{li}\\
       \vdots\\
      \tilde c_{Ni}
   \end{bmatrix}.\label{eqn:XiandC}
\end{align}
Putting $\tensor \Xi_i$ and $\tilde{\h c}_i$ as the $i$th column of the matrices $\tensor \Xi$ and $\tilde{\tensor C}$, respectively, we can solve $\tilde{\tensor C}$ from $\tensor \Phi^N \tilde{\tensor C} = \tensor \Xi$. As $\bm\xi_1=\bm x_0$ by construction, $\vphi_j^N(\bm x_0)$ is the first element of vector $ \bvphin_j$, denoted by $(\bvphin_j)_1$.
Thus, 
the solution of the dynamical system  $\dot{\h x}( t) = \h u(\h x)$ can be approximated by 
\begin{align}
    \label{eqn:sol}
    \bm x(t) = \sum_{j = 1}^N \tilde {\h c}_j \vphin_j(\bm x_0) e^{\tilde \lambda_j t} = \sum_{j = 1}^N \tilde {\h c}_j ( \bvphin_j)_1 e^{\tilde \lambda_j t}.
\end{align}

\section{The proposed approach}
\label{sect:proposed}


We regard a critical point of the optimization problem \eqref{eq:setup}  as a steady state of the following dynamical system,
\begin{align}
    \begin{cases}
    \dot{\h x}( t) = \h u(\h x)\\
    \h x(t_0) = \h x_{0},
    \end{cases}\label{eq:Dynamicl system}
\end{align}
where $\h u(\h x) := -\nabla f(\h x)$   and $\h x_0$ is an initial point. In other words, we find a critical point of \eqref{eq:setup} by evolving the corresponding dynamical system \eqref{eq:Dynamicl system} for a sufficiently long time to arrive at the equilibrium that satisfies $\nabla f(\h x) = \h 0$ . 
As for the min-max problem \eqref{eq:minmax setup}, its corresponding dynamical system is
\begin{align}
    \begin{cases}
    \dot{\h x}( t) = -\nabla_{\h x}f(\h x, \h y)\\
    \dot{\h y}( t) = \nabla_{\h y}f(\h x, \h y)\\
    \h x(t_0) = \h x_{0}\\
    \h y(t_0) = \h y_{0},
    \end{cases}\label{eq:DSforminmax}
\end{align}
which can be written in terms of \eqref{eq:Dynamicl system} with $\bar{\h x}=(\h x, \h y)^{\trans}$ and $\h u(\bar{\h x}) = (-\nabla_{\h x} f(\h x,\h y), \nabla_{\h y} f(\h x,\h y))^{\trans}$.  Consequently, we only focus on the general optimization problem \eqref{eq:setup} and its corresponding dynamical system \eqref{eq:Dynamicl system}.

If the dynamical system \eqref{eq:Dynamicl system} is linear, one can  express it as
\begin{align}
    \begin{cases}
    \dot{\h x} (t) = A \h x(t) \\
    \h x(0) = \h x_{0},
    \end{cases}\label{eq:linear system}
\end{align}
where $A$  is a finite-dimensional time invariant matrix. There is a closed-form solution of \eqref{eq:linear system} given by Duhamel's formula \cite{teschl2012ordinary}, i.e.,
\begin{align}
    \h x(t) = \exp(tA)\h x_{0}. 
    \label{eq:duhamel}
\end{align}
Motivated by the fact that a linear dynamical system has a closed-form solution, we propose to employ the adaptive spectral Koopman (ASK) method \cite{li2023adaptive} that robustly linearizes an arbitrary dynamical system with its steady state serving as a critical point of the corresponding optimization problem. 

Before presenting our approach, we highlight the differences between solving ODEs and optimization problems. First, unlike ODE evolution, which closely tracks trajectories, optimization algorithms target the steady state—local minima for minimization problems and saddle points in min-max problems. This distinction affects parameter selection, such as the order/level of approximation (i.e., $N$) and length of evolution time (i.e., $T$). The choice of these parameters in optimization is more flexible than in ODEs, as the optimization focuses on the long-term behaviors of the system rather than precise trajectory tracking. This flexibility allows for a rough approximation of the dynamics in a neighborhood of the current state using mostly linear functions instead of higher-order polynomials. Second, it enables us to extend the time window for numerical integration without being overly concerned about errors introduced by the truncation in the finite-dimensional approximation.
Our approach takes a distinct perspective, aiming to stably evolve the system towards equilibrium in the long term.




In \Cref{subsec:ASK_solution_construction}, we describe our workflow of finding a steady state of the dynamical system \eqref{eq:Dynamicl system} via ASK, highlighting important steps discussed in \Cref{sect:ASKreview} and our modifications. Theoretical analysis on error bounds is conducted in \Cref{sect:anal}.

\subsection{Workflow}
\label{subsec:ASK_solution_construction}

We consider each component of the dynamics $\h u$ in \eqref{eq:Dynamicl system} as an observable $g$ that defines the Koopman operator \eqref{eq:koopman}, thus leading to a system of observables \eqref{eq:sys-obs}. To leverage the properties of the Koopman operator, 
 we consider its finite-dimensional approximation \eqref{eqn:trun_expansion}  and hence the solution of \eqref{eq:Dynamicl system} can be obtained by a linear combination of the approximated eigenfunctions $\{\varphi_j^N\}$ with corresponding eigenvalues $\{\tilde{\lambda}_j\}$. As $t\rightarrow \infty,$ the solution $\h x(t)$ given by \eqref{eqn:sol} approaches a critical point to the optimization problem \eqref{eq:setup}. We choose $t=T$ for a sufficiently large value of $T$. 

However, the accuracy of the solution \eqref{eqn:sol} may deteriorate due to the finite-dimensional approximation of $\mathcal{K}$ and the local approximation of $\varphi$ in the vicinity of $\bm x_0$, which is notably pronounced in systems characterized by highly nonlinear dynamics. To mitigate this issue, we periodically update 
an acceptable range for each component
\begin{equation}
R_i := \left[L_i , U_i \right],
\label{eq:acceptable range}
\end{equation}
where $L_i$ and $U_i$ symbolize the lower and upper bounds, respectively, and $r$ denotes a flexible radius parameter. The interval $[L_i, U_i]$ determines the construction of grid points $\{\bm \xi_l \}_{l=1}^N$ used in \eqref{eqn:sparse_grid_interp}. In particular, we construct the \textit{sparse} grid points \cite{li2022sparse} on the reference domain $[-1, 1]$ to define the matrices $\{\tensor
G_i\}$ in \eqref{eqn:Gi} and $\{\tensor \Xi_i\}$ in \eqref{eqn:XiandC}, which can then be rescaled to  match the real domain $[L_i, U_i]$ by $\frac{2 \tensor G_i}{U_i - L_i}$ and  $\frac{U_i - L_i}{2} (\tensor \Xi_i + 1) + L_i$, respectively, for $i=1, \cdots, d.$ By construction,
the central point of the domain is the first generated sparse grid point, e.g.,  $(0,0,\dotsc,0) = \bm \xi_1$ for a multidimensional domain $[-1, 1]^d$. 
In order to avoid the situation of $\bm x_0 \notin \{\bm \xi_l\}$, we consider a neighborhood of $\bm x_0$ defined by $[\bm x_0 - \bm r, \bm x_0 + \bm r]$, where $\bm r = (r_1, r_2, \dotsc, r_d)\trans$ is the radius of the neighbor so that the center of the neighborhood is always within the sparse grid.
For simplicity, we adopt the isotropic setting, i.e., $r_1 = r_2 = \cdots = r_d = r$.

Initially, we set $L_i = x_i - r$ and $U_i = x_i + r$. 
For the current state, denoted as $\bm x(t) = \left( x_1(t), x_2(t), \dotsc, x_d(t) \right)\trans$, we define the neighborhood $R_1 \times R_2 \times \dotsc \times R_d$ as valid if all components $x_i(t)$ fall within their respective ranges $R_i$. However, should any component lie outside its range, i.e., $x_i(t) \notin R_i$, we implement a straightforward retraction procedure by halving $t$ i.e. $t=\frac{t}{2}$.
%
Our choice for the retraction by $t=\frac{t}{2}$ is primarily motivated by its simple implementation, but other methods are worth exploring. 
For instance, we can leverage bifurcation concepts to extend \(T\) until it approaches the boundary of one of the intervals \(R_i\). This approach harnesses the full potential of the \comm{local dynamics } dynamics in the neighborhood, but it does not guarantee improved performance compared to a straightforward bifurcation approach. This is because the error introduced by numerical integration tends to be larger at the boundaries of the interval, and the range of feasible parameters is narrower in such cases.




By setting the end state of the current iteration as the initial condition to the next one,  we rescale $\{\tensor
G_i\}$  and $\{\tensor \Xi_i\}$ that are defined on $[-1, 1]$ by
$\frac{2 \tensor G_i}{U_i - L_i}$ and  $\frac{U_i - L_i}{2} (\tensor \Xi_i + 1) + L_i$ to deal with the new interval $[L_i, U_i]$, thus getting a new solution \eqref{eqn:sol} after eigendecomposition and adjustment to make the obtained solution within the bounds $[L_i, U_i]$. We stop the iterations until a stopping criterion is met or the maximum iteration number is achieved. The overall algorithm is summarized in~\cref{algo:ASK_pseudo_code}.




\begin{algorithm}[t]
    \caption{Sparse Grid Adaptive Spectral Koopman Method}
    \label{algo:ASK_pseudo_code}
    \begin{algorithmic}[1]
        \REQUIRE $n, T, \bm x_{0}, r, \kappa, \epsilon$
        \STATE {Set $k=0, \h x^{(0)} = \h x_0$.}
        \STATE{Generate sparse grid points $\{\bm \xi_l \}_{l=1}^N$}
        \WHILE{$k<n$ or $\|\h u(\bm x^{(k)})\| > \epsilon$}
            \STATE{Let $L_i=x_{i}-r, U_i = x_{i}+r$, set neighborhood $R_i$ as $R_i=[L_i, U_i]$ for $i = 1,2,...,d$.}
            \STATE{Compute $\tensor G_i$ at $\h x^{(k)}$ in \eqref{eqn:Gi} to obtain $\tensor U = \sum_i^d \: \text{diag}{\big( g_i(\bm \xi_1), \dotsc, g_i(\bm \xi_N) \big)} \: \tensor G_i$.}
            \STATE{Given $\tensor M$ in \eqref{eqn:sparse_grid_interp} and $\tensor U$, apply eigen-decomposition $\tensor U \tensor W = \tensor M \tensor W \tensor \Lambda$ to get $\tensor \Phi^N = \tensor M \tensor W$.}
            \STATE{Solve linear system $\tensor \Phi^N \tensor C = \tensor \Xi$, where $\tensor \Xi$ is defined in \eqref{eqn:XiandC}.}
            \STATE{Let $t = T$.}
            \STATE{Let $\nu_j$ be the first element of the $j$th column of $\tensor \Phi$. Construct solution at $t$ as $$\bm x(t) = \displaystyle\sum_j \tensor C(j, :) \nu_j \exp\left({\tilde \lambda_j (t)}\right),$$ where $\tensor C(j, :)$ is the $j$th row of $\tensor C$ and $\tilde \lambda_j$ is the $j$th eigenvalue of $\tensor\Lambda.$}
            \WHILE{$x_i(t) \notin R_i$ for any $i$}
                \STATE{ Set t = t/2.}
                \STATE{Re-evaluate  $\bm x(t) = \displaystyle\sum_j \tensor C(j, :) \nu_j \exp\left({\tilde \lambda_j (t)}\right)$.}       
            \ENDWHILE
            \STATE Evolve the current solution to the time $t$, i.e., $\h x^{(k)} = \h x(t)$
            \STATE $k = k+1$.
        \ENDWHILE
        \RETURN $\bm x^{(k)}$
    \end{algorithmic}
\end{algorithm}

Note that we approximate the eigenfunctions using multi-dimensional orthogonal polynomials through the spectral method. Consequently, the Koopman operator's generator is discretized as a differentiation matrix. Similar approaches have been proposed for studying the behavior of differential equations~\cite{breda2006pseudospectral, breda2012approximation, froyland2013estimating}. Furthermore, these eigenfunctions can be constructed using various approximation approaches. For example, they can be built based on the kernel of a reproducing kernel Hilbert space~\cite{das2020koopman, giannakis2019data} in the data-driven setting. A more in-depth exploration of the properties of the Koopman operator and its applications to differential equations can be found, e.g., in~\cite{mezic2005spectral, arbabi2017ergodic,korda2020data, nakao2020spectral, nathan2018applied, brunton2021modern, williams2015data, williams2016extending, dietrich2020koopman, balabane2021koopman}. Certain methods discussed therein hold the potential to enhance the efficiency and accuracy of our approach, and these will be incorporated into our future work.

\input{analysis}
\section{Numerical Results}\label{sect:experiment}
We present extensive numerical examples to
demonstrate the performance of the proposed ASK method for solving minimization problems \eqref{eq:setup} in Section~\ref{sect:exp-min} and min-max problems \eqref{eq:minmax setup} in Section~\ref{sect:exp-minmax}.  All of our numerical experiments are conducted using 100 randomly chosen initial points. We assess the efficiency of each competing algorithm based on both $\|\nabla f (\h x)\|$ and time consumption. For convex test functions, we include the comparison to the CVX package~\cite{gb08,cvx}, using the error to the global minimum as an additional evaluation metric. 
All the experiments are performed in MATLAB 2022a. We refer to a textbook  \cite{watkins2004fundamentals} to implement these baseline algorithms in MATLAB. No other software or public libraries were used in our experiments. The computing platform is a personal computer with Ryzen 7 5800x, 32 GB of RAM, and the operating system of Windows 11.

For the proposed method, we adopt the maximum iteration  (i.e., adaptivity steps) number of $5 \times 10^4$. There are two major parameters that need to be tuned: one is the radius $r$ of the acceptable range defined in \eqref{eq:acceptable range}, and the other is the maximum degree of basis polynomial function $\Psi,$ which we refer to as the level of approximation. The acceptable range $r$ is set to $10^{-1}$ and the level of approximation in the sparse grids method is set to $1$ if not mentioned otherwise. In general, for level $l$, we generate $2^l+1$ one-dimensional quadrature points and then use them to construct high-dimensional sparse grids points based the Smolyak structure. A method is considered to have successfully converged if it achieves a gradient norm $\|\nabla f (\h x)\|$ less than $10^{-6}$. The reported results reflect the average outcomes over all such successful trials.



\subsection{Minimization problems}\label{sect:exp-min}
We investigate bowl-shaped functions in \Cref{sect:Bowl-Shaped}
and valley-shaped functions in \Cref{sect:valley-shaped} focusing on 2-dimensional functions. In addition, 
two higher-dimensional examples are discussed in \Cref{sect:HD}. All these functions are widely used in optimization for testing purposes. We compare the proposed ASK method with three baseline approaches: gradient descent (GD) in \eqref{eq:gradient descend}, heavy ball (HB) in \eqref{eq:grad momentum}, and Nesterov’s accelerated gradient (NAG) in \eqref{eq:Nest accel grad}.

\subsubsection{Bowl-shaped testing functions}\label{sect:Bowl-Shaped}

\begin{figure}
    \centering
    \subfigure[Trajectory of ASK method]{
    \includegraphics[width=0.4\textwidth]{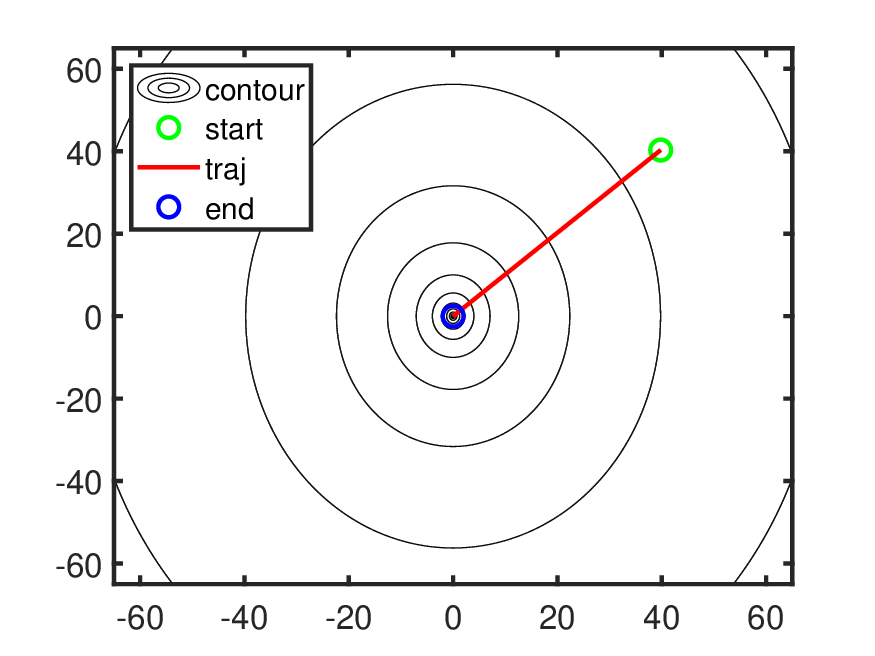}
    \label{fig:RothypTrace}}
    \subfigure[$\|\nabla f(\bm x^{(k)}) \|$ v.s.~$k$.]{
    \includegraphics[width=0.4\textwidth]{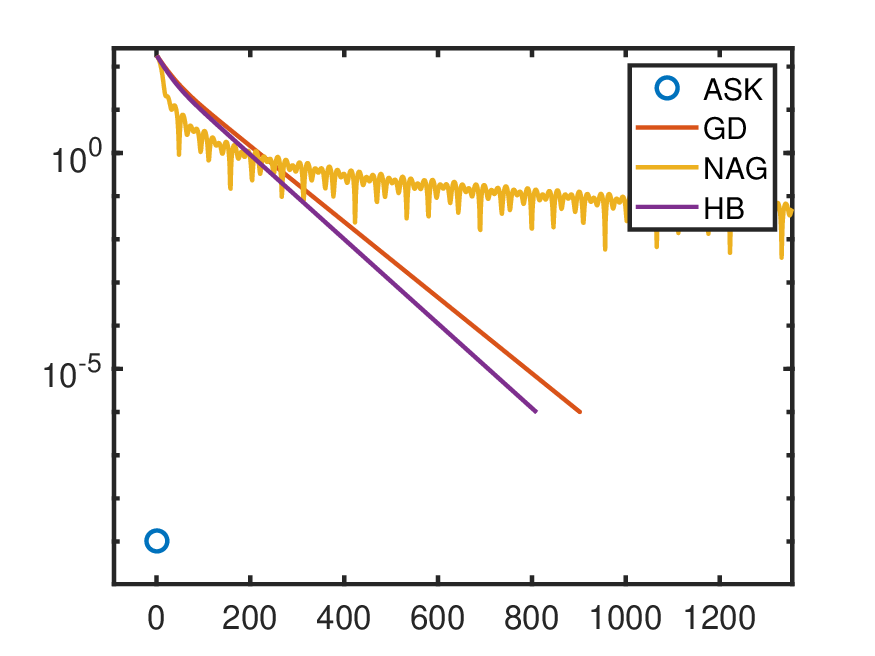}
    \label{fig:RothypGrad}}
    \caption{Algorithmic behaviors on Rotated Hyper-Ellipsoid function.
    As a semi-continuous algorithm, the proposed ASK method is able to reach the global optimum in one iteration.}
    \label{fig:Rothyp}
\end{figure}

\begin{table}
    \centering
    \begin{tabular}{cccc}
    \hline\hline
         Methods& $\|\nabla f\|$& Error & Time\\
         \hline
         ASK&1.0303e-09&3.6223e-10&4.8347e-03 \\
         GD&9.8228e-07&4.8131e-07&8.0462e-03 \\
         NAG&2.5607e-03&1.0887e-03&6.6670e-02 \\
         HB&9.8054e-07&4.7935e-07&6.4078e-03 \\
         CVX&3.8696e-28&9.6740e-29&2.3438e-01 \\
         \hline\hline
    \end{tabular}
    \caption{Rotated Hyper-Ellipsoid, init=$130\times\text{rand}(2,1)-65$.}
    \label{tab:Rothyp}
\end{table}

We start with the rotated Hyper-Ellipsoid function \cite{molga2016test}, which is defined by 
\begin{align}
    f(x_1, x_2) = \sum_{i=1}^{2}\sum_{j=1}^{i} x_{j}^{2}.
\end{align}
This function is convex and has a unique minimum of $f(\h x^{*}) = 0$ at $\h x^{*} = (0,0)$. 
An initial point is randomly chosen within the grid $[-65,65]\times[-65,65]$.
\Cref{fig:RothypTrace} illustrates a notable advantage of the proposed algorithm: 
it converges in just one iteration. 
This example showcases a scenario where local behaviours of the dynamics closely align with the global dynamics, enabling the identification of the global minimum in a single iteration. It also emphasizes the remarkable potential of our proposed method when applied in such circumstances.
Figure \ref{fig:RothypGrad} depicts the norm of the gradient at each iteration,  confirming once again that the ASK algorithm converges in a single iteration. In addition, NAG demonstrates faster convergence than GD and HB initially but slows down as it approaches the target location.
The quantitative comparison Table~\ref{tab:Rothyp} indicates that all competing algorithms successfully find the global minimum. Our proposed algorithm converges closer to the global minimum than the three gradient-based algorithms with the least amount of computational time. While the CVX package achieves the highest accuracy, it is also the slowest.

The second testing function  is the sum of different powers  \cite{molga2016test}, given by
\begin{align}
    f(x_1, x_2) = \sum_{i=1}^{2} |x_{i}|^{i+1}.
\end{align}
It  has a unique minimum of $f(\h x^{*}) = 0$ at $\h x^* = (0,0)$. Due to the presence of $|x_2|^3,$ the function is not differentiable. 
We choose an initial point randomly from $[-1,1]\times[-1,1]$. 
\Cref{fig:SumpowTrace} illustrates that the evolution of our proposed algorithm is relatively smooth, despite the function itself not being smooth, specifically, not differentiable.
Both \Cref{fig:Sumpow} and  \Cref{tab:Sumpow}  demonstrate that the proposed algorithm achieves the highest accuracy in terms of the norm of the gradient and error to the global solution with the least amount of computation time.  The result obtained by CVX is included, showing that CVX yields the exact solution but at the cost of a longer runtime. 


\begin{figure}
    \centering
    \subfigure[Trajectory of ASK method]{
    \includegraphics[width=0.4\textwidth]{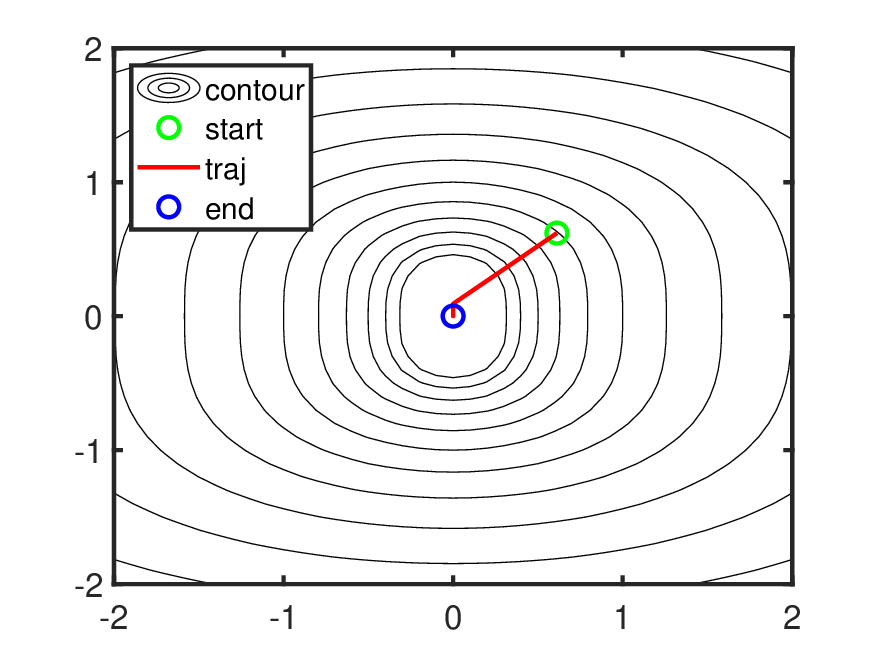}
    \label{fig:SumpowTrace}}
    \subfigure[$\|\nabla f(\bm x^{(k)}) \|$ v.s.~$k$.]{
    \includegraphics[width=0.4\textwidth]{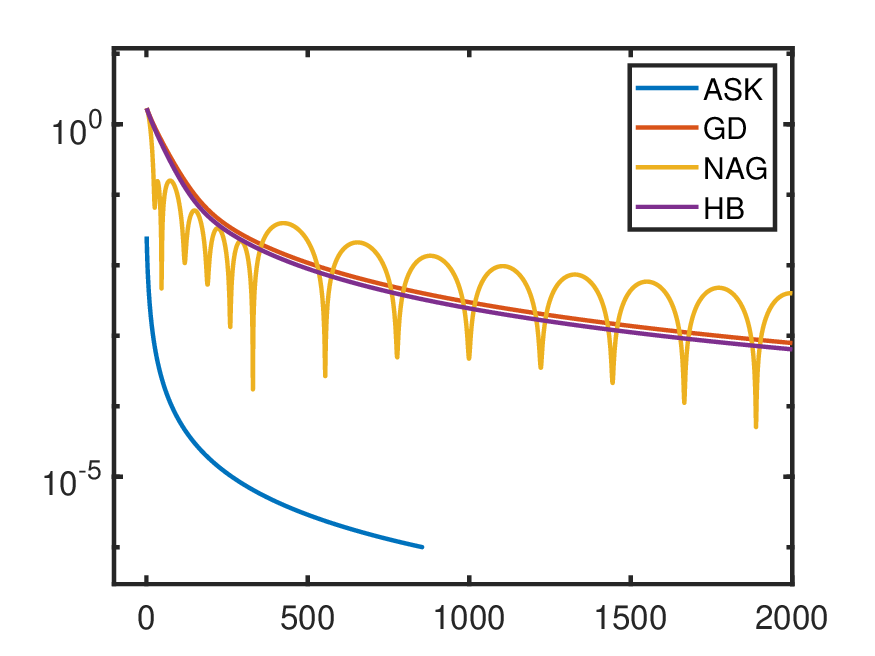}
    \label{fig:SumpowGrad}}
    \caption{Algorithmic behaviors on the sum of different powers function. 
    The proposed algorithm converges the fastest to the solution with the smallest norm of the gradient. Despite the objective function is not smooth, the evolution of ASK is rather smooth.}
    \label{fig:Sumpow}
\end{figure}

\begin{table}
    \centering
    \begin{tabular}{cccc}
    \hline\hline
         Methods& $\|\nabla f\|$& Error  & Time\\
         \hline
         ASK&9.9966e-07&5.7725e-04&7.8031e-02 \\
         GD&3.2950e-05&3.3138e-03&1.0769e-01 \\
         NAG&2.2653e-03&3.9177e-03&7.0358e-02 \\
         HB&2.6712e-05&2.9836e-03&7.0009e-02 \\
         CVX&0.0000e+00&0.0000e+00&4.3750e-01 \\
         \hline\hline
    \end{tabular}
    \caption{The sum of different powers function, init=$2\times\text{rand}(2,1)-1$.}
    \label{tab:Sumpow}
\end{table}

Lastly, we consider the second Bohachevsky function \cite{bohachevsky1986generalized}, labeled as Bohachevsky 2. It is defined by
\begin{align}
    f(x_1, x_2) = x_{1}^{2} + x_{2}^{2} - 0.3\cos(3\pi x_{1})\cos(4\pi x_{2}) + 0.3.
\end{align}
As illustrated in \Cref{fig:Bohachevsky2trace}, the function is non-convex and takes the form of a bowl with oscillatory level sets near the global minimum. Specifically, it features numerous local minima and a global minimum of $f(\h x^{*}) = 0$ at $\h x^{*} = (0,0)$. For this function, the level of approximation is set to 3 rather than the default 1 due to the intricate dynamics around the critical points. We randomly choose initial points in the domain $[-5,5]\times[-5,5]$. 
When an initial point is far from a local extreme cluster, the proposed algorithm converges to a local minimum. 

\Cref{fig:Bohachevsky2trace} displays the trace of the proposed ASK method on the contour plot, revealing a zigzag path when approaching a critical point. 
\Cref{fig:Bohachevsky2grad} compares the competing algorithms in terms of the norm of the gradient at each iteration, indicating that all the methods converge to a critical point except NAG. This is anticipated, as the NAG method may have adverse effects when applied to a non-convex function~\cite{wibisono2016variational}. The proposed ASK achieves the highest accuracy among the competing methods.  \Cref{tab:Bohachevsky2} confirms that ASK results in the smallest gradient norm, yet, unfortunately, it is the slowest due to the intricate \comm{local dynamics} dynamic in the neighbourhood that necessitate numerous eigen-decomposition calculations. This behavior is observed in \Cref{fig:Bohachevsky2grad}, where the magnitude of $\|\nabla f \|$ oscillates between $10^{-1}$ and $10^{1}$ before rapidly dropping below $10^{-10}$ after approximately 100 iterations.


\begin{figure}
    \centering
    \subfigure[Trajectory of ASK method]{
    \includegraphics[width=0.4\textwidth]{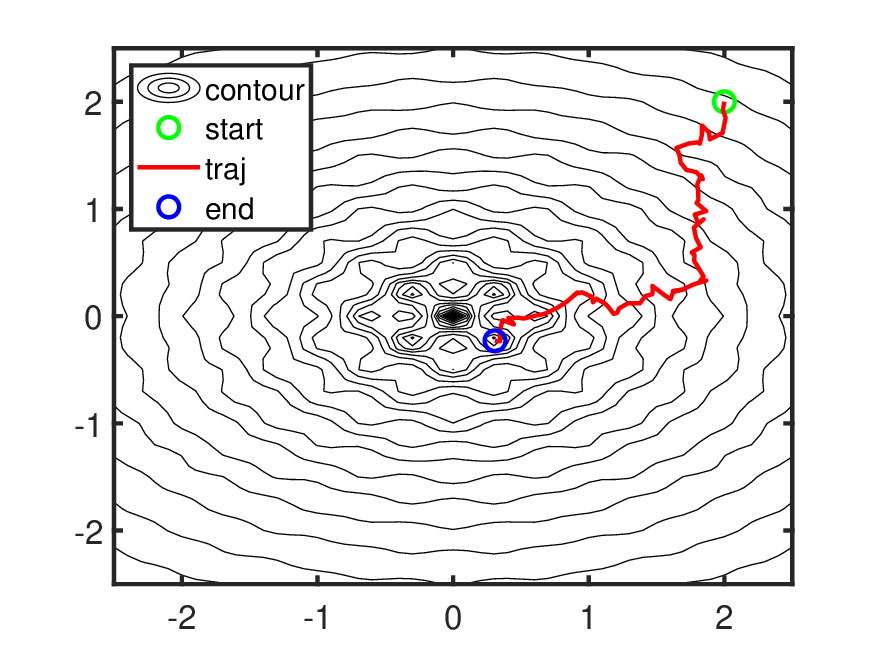}
    \label{fig:Bohachevsky2trace}
    }
    \subfigure[$\|\nabla f(\bm x^{(k)}) \|$ v.s.~$k$.]{\includegraphics[width=0.4\textwidth]{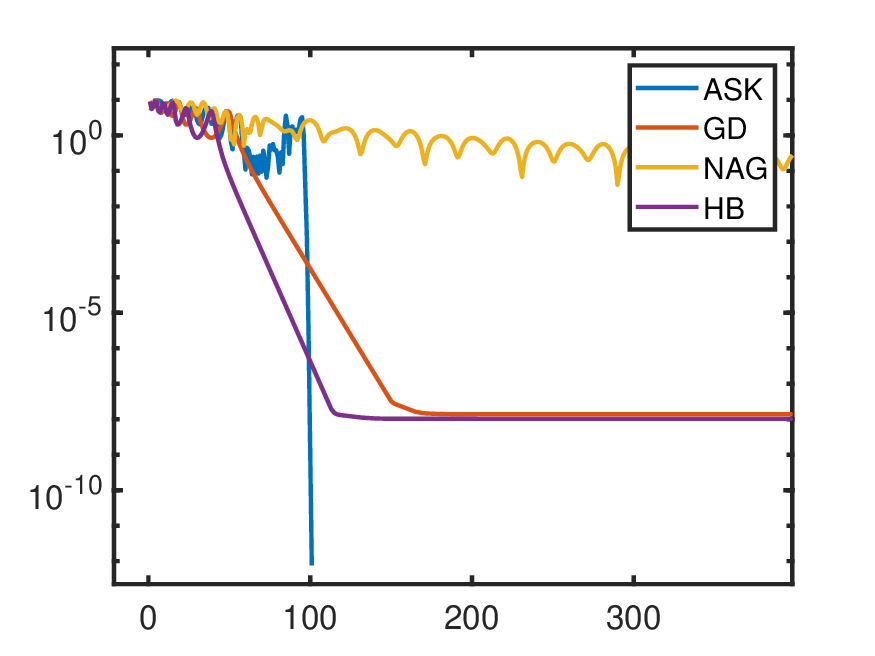}
   \label{fig:Bohachevsky2grad}
    }      
    \caption{Algorithmic behaviors on Bohachevsky 2 function. All the methods converge to a critical point except for NAG, and ASK achieves the highest accuracy in terms of the gradient norm.}
    \label{fig:Bohachevsky2}
\end{figure}

\begin{table}
    \centering
    \begin{tabular}{ccc}
    \hline\hline
         Methods & $\|\nabla f\|$ & Time \\
         \hline
         ASK & \textbf{3.7616e-14} & 1.1111e-01 \\
         GD & 5.2897e-09 & 2.1634e-03 \\
         NAG & 1.4387e-01 & 6.2574e-03 \\
         HB & 9.5994e-10 & \textbf{1.1356e-03} \\
    \hline\hline
    \end{tabular}
    \caption{Bohachevsky 2 function, init=$4\times\text{rand}(2,1)-2$.}
    \label{tab:Bohachevsky2}
\end{table}

\subsubsection{Valley-Shaped functions}\label{sect:valley-shaped}
The first valley-shaped testing function we consider is called the three-hump Camel function \cite{molga2016test}, defined by
\begin{align}\label{eq:camel3}
    f(x_1, x_2) = 2x_{1}^{2} - 1.05x_{1}^4+\frac{x_{1}^{6}}{6} + x_{1}x_{2} + x_{2}^{2}.
\end{align}
This nonconvex function has a global minimum of $f(\h x^{*}) = 0$ at $\h x^{*} = (0,0)$, two local minima, and two saddle points. 
The initial point is chosen randomly from $[-5,5]\times[-5,5]$. 
As shown in \Cref{fig:Camel3} and \Cref{tab:Camel3}, the proposed algorithm efficiently approaches a critical point in just a few iterations, making it the most time-effective. Additionally, it outperforms other competing methods in terms of the gradient norm. 


\begin{figure}
    \centering
    \subfigure[Trajectory of ASK method]{
    \includegraphics[width=0.4\textwidth]{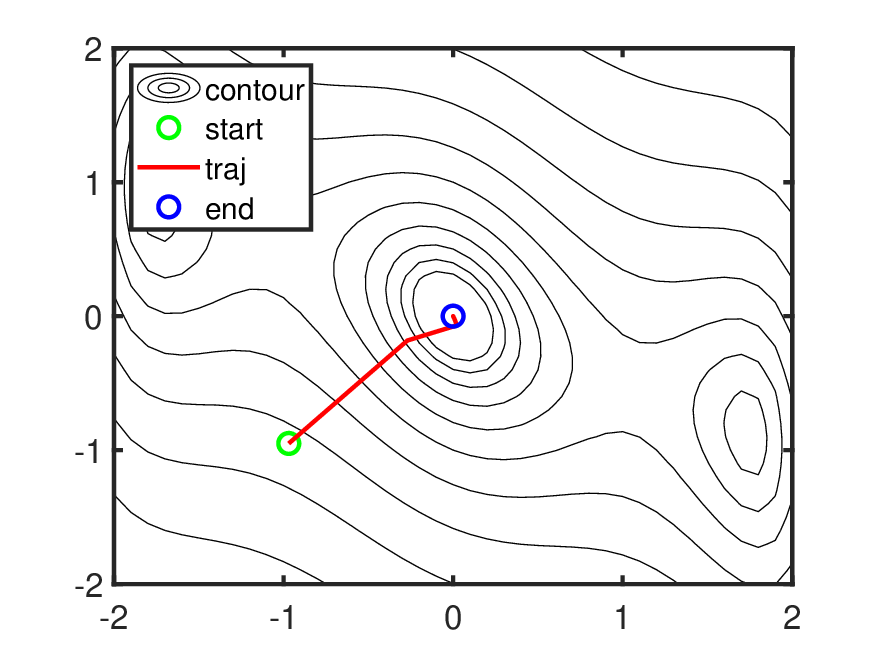}
    \label{fig:camel3Trace}}
    \subfigure[$\|\nabla f(\bm x^{(k)}) \|$ v.s.~$k$.]{
    \includegraphics[width = 0.4\textwidth]{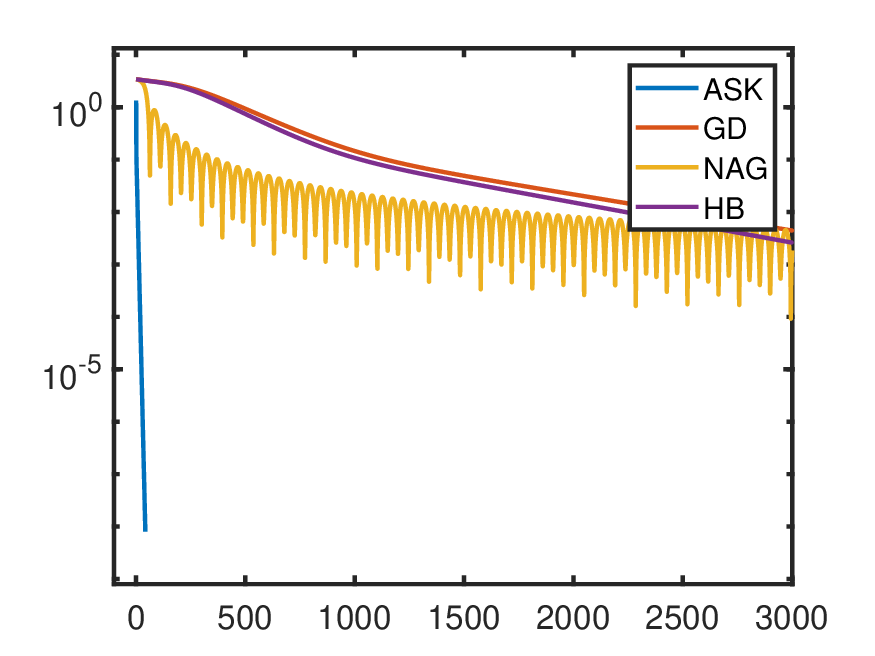}
    \label{fig:camel3Grad}}    
    \caption{Algorithmic behaviors on the three-hump Camel function. 
    The proposed algorithm converges the fastest to a critical point.}
    \label{fig:Camel3}
\end{figure}

\begin{table}
    \centering
    \begin{tabular}{ccc}
    \hline\hline
         Methods & $\|\nabla f\|$ & Time \\
         \hline
         ASK &\textbf{ 7.9837e-09} & \textbf{8.0243e-03} \\
         GD & 6.6493e-08 & 7.9303e-02 \\
         NAG & 4.5331e-04 & 5.6113e-02 \\
         HB & 1.1336e-08 & 5.5693e-02 \\
    \hline\hline
    \end{tabular}
    \caption{Three-hump Camel function, init=$10\times\text{rand}(2,1)-5$.}
    \label{tab:Camel3}
\end{table}

We then consider a six-hump Camel function \cite{hedar2007global}, defined by
\begin{align}
    f(x_1, x_2) = \left(4 - 2.1 x_{1}^{2} + \frac{x_{1}^{4}}{3}\right)x_{1}^{2} + x_{1}x_{2} + (-4 + 4x_{2}^{2})x_{2}^{2},
\end{align}
as a testing function. This function is non-convex and has a global minimum of approximately $f(\h x^{*}) = -1.0316$ at $(0.0898,-0.7126)$ and $(-0.0898,0.7126)$, along with four other local minima. 
We select  random initial points from $[-3,3]\times[-2,2]$ and present the results in \Cref{fig:Camel6} and \Cref{tab:Camel6}.
\Cref{fig:Camel6} illustrates that the proposed algorithm requires only a few iterations to converge. \Cref{tab:Camel6} indicates that ASK slightly outperforms other methods in terms of the gradient norm. While all the methods have comparable computation times, as listed in \Cref{tab:Camel6}, the complexity per ASK iteration is rather high due to the eigen-decomposition.

\begin{figure}
    \centering
    \subfigure[Trajectory of ASK method]{
    \includegraphics[width=0.4\textwidth]{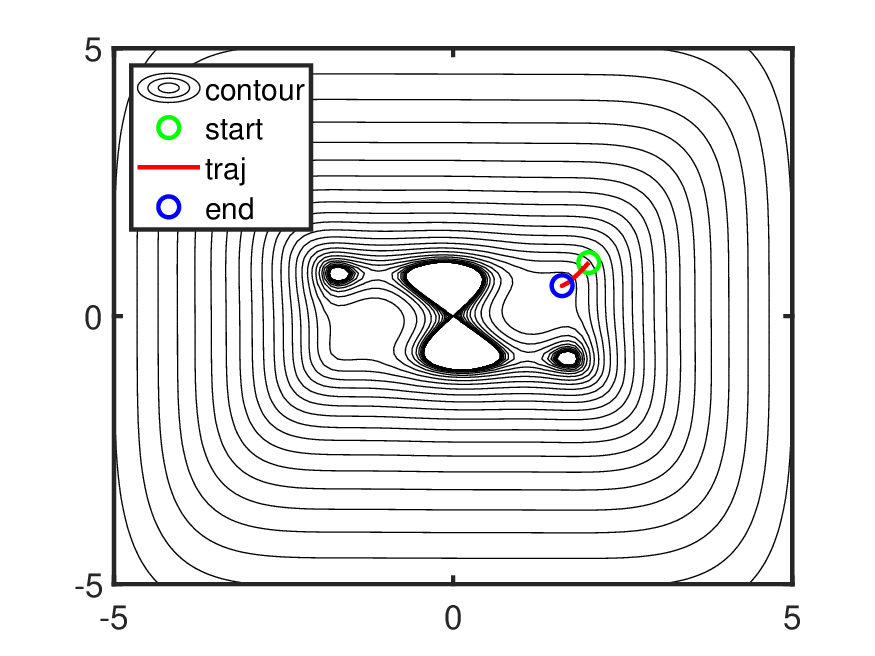}
    \label{fig:camel6Trace}}
    \subfigure[$\|\nabla f(\bm x^{(k)}) \|$ v.s.~$k$.]{
    \includegraphics[width=0.4\textwidth]{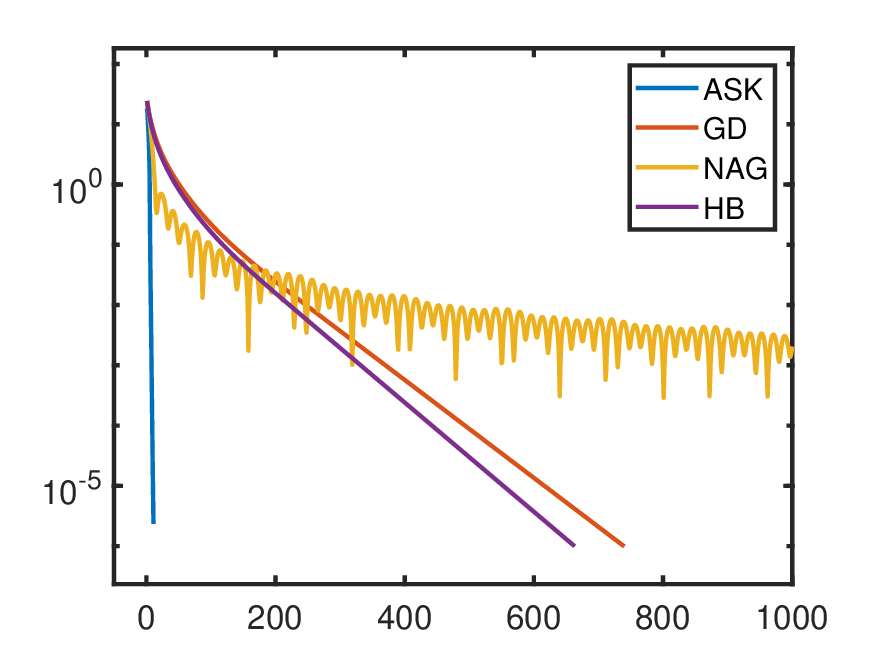}
    \label{fig:camel6Grad}}
    \caption{Algorithmic behaviors on the six-hump Camel function. 
    The proposed algorithm converges the fastest to a critical point.}
    \label{fig:Camel6}
\end{figure}

\begin{table}
    \centering
    \begin{tabular}{ccc}
    \hline\hline
         Methods & $\|\nabla f\|$ & Time \\
         \hline
         ASK & 5.3550e-07 & 9.3056e-03 \\
         GD & 9.9561e-07 & 7.4152e-03 \\
         NAG & 2.9041e-03 & 6.5248e-03 \\
         HB & 9.8639e-07 & 5.1908e-03 \\
    \hline\hline
    \end{tabular}
    \caption{Six-hump Camel Function, init=$6\times\text{rand}(2,1)-3$. }
    \label{tab:Camel6}
\end{table}

The third valley-shaped testing function  is the Dixon-Price function \cite{hedar2007global,jamil2013literature}, defined by
\begin{align}
    f(\h x) = (x_{1} - 1)^{2} +\sum_{j=2}^{d} j(2x_{j}^{2} - x_{j-1})^2.
    \label{eq:Dixonpr}
\end{align}
The function has a global minimum of $f(\h x^{*})= 0$ at $x_{j}^{*} = 2^{-\frac{2^{j}-2}{2^{j}}}$ for $j=1,\cdots, d$. We examine the case of $d = 2$ here,  $d=10$ in \Cref{sect:HD}, and set $d = 100$ as a min-max problem in \Cref{sect:exp-minmax}.
We randomly choose initial points from $[-10,10]\times[-10,10]$.  \Cref{tab:Dixonpr} demonstrates that our proposed method achieves nearly three orders of magnitude improvement over GD, while being three orders of magnitude faster. NAG does not converge to any critical point, as indicated by a large gradient norm.
The HB method achieves a result similar to ASK but takes ten times longer.
 \Cref{fig:Dixonpr} aligns with \Cref{tab:Dixonpr}, showing that ASK converges in a few iterations, while NAG  diverges.


\begin{table} 
    \centering
    \begin{tabular}{ccc}
    \hline\hline
         Methods & $\|\nabla f\|$ & Time \\
         \hline
         ASK & \textbf{9.4133e-08} & \textbf{6.7219e-03}\\
         GD & 1.5514e-04 & 1.1269e+00 \\
         NAG & 3.1235e-01 & 4.3161e-02 \\
         HB & 5.6322e-07 & 9.7755e-02 \\
    \hline\hline
    \end{tabular}
    \caption{Dixon-Price function, init=$20\times\text{rand}(2,1)-10$.}
    \label{tab:Dixonpr}
\end{table}

\begin{figure} 
    \centering
    \subfigure[Trajectory of ASK method]{
    \includegraphics[width=0.4\textwidth]{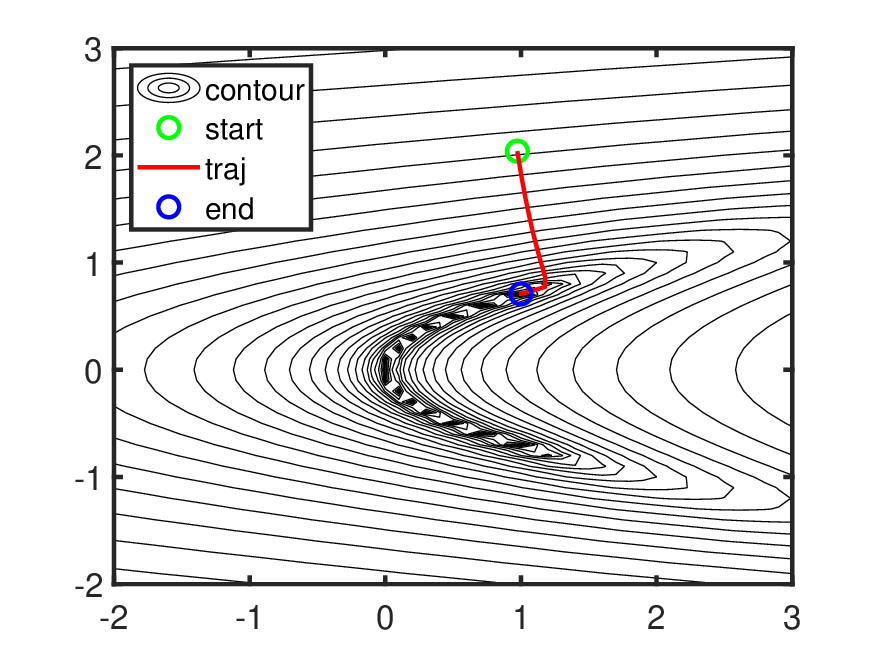}
    \label{fig:DixonprTrace}}
    \subfigure[$\|\nabla f(\bm x^{(k)}) \|$ v.s.~$k$.]{
    \includegraphics[width=0.4\textwidth]{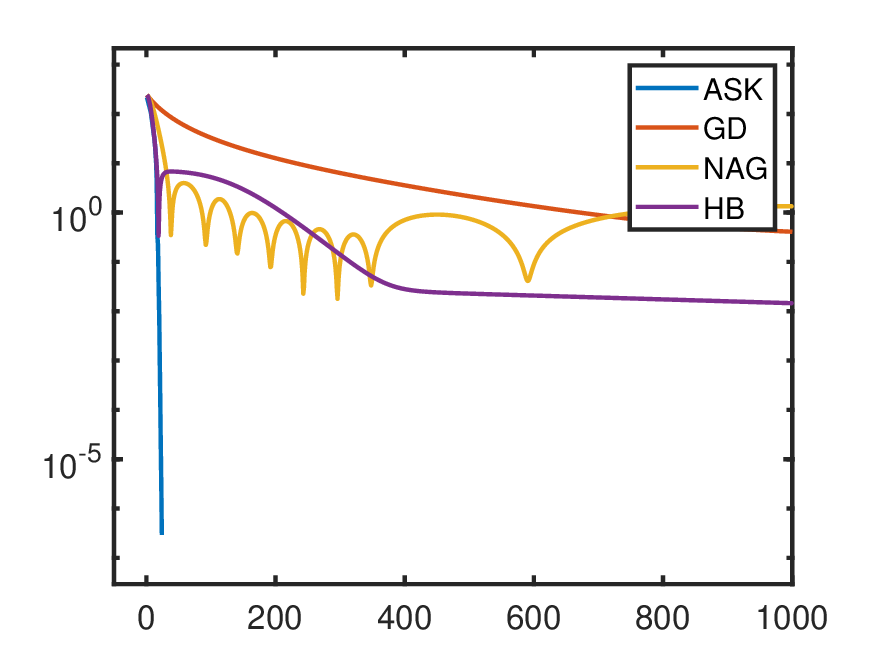}
    \label{fig:DixonprGrad}}
    \caption{Algorithmic behaviors on the Dixon-Price function. 
    The proposed algorithm converges to a critical point with a few iterations, while NAG seems diverges in the end.}
    \label{fig:Dixonpr}
\end{figure}

The last valley-shaped testing function is the Rosenbrock function \cite{rosenbrock1960automatic,hedar2007global} given as 
\begin{align}
    f(x_1, x_2) =100(x_{2} - x_{1}^{2})^{2} + (x_{1} - 1 )^{2}.
    \label{eq:rosen}
\end{align}
This function has a global minimum of $f(\h x^{*}) = 0$ at $\h x^{*} = (1,1)$. 
We start with a random initial point in $[-2,2]\times[-2,2]$. 
The quantitative behavior is presented in  \Cref{tab:rosenbrock}, showing that ASK yields the highest accuracy in terms of the gradient norm with the least amount of time. 
 \Cref{fig:RosenTrace} illustrates the trajectory of the ASK method on a contour plot, while
  \Cref{fig:RosenGrad} reveals a barrier in the path where the gradient's norm is between $10^{-1}$ and $10^{-2}$, signifying that all the competing methods significantly slow down. It takes ASK about 100 iterations to overcome this barrier. 
Naive GD and HB methods are able to pass through this barrier, while NAG fails in this case. We postulate that this phenomenon might stem from the intricate dynamics surrounding the crescent-shaped non-convex set where local minima are located.
Additionally, we note that we set the level of approximation to 3 for this example, rather than the default choice of 1, as we believe a higher-order approximation is required to achieve better results.



\begin{table}
    \centering
    \begin{tabular}{ccc}
    \hline\hline
     Methods & $\|\nabla f\|$ & Time\\
     \hline
     ASK&\textbf{3.0836e-07} & \textbf{2.6915e-02} \\
     GD&5.4417e-05 & 7.4911e-01 \\
     NAG&5.9793e-02 & 1.0775e-01 \\
     HB&1.2757e-05 & 2.9323e-01 \\
     \hline\hline
    \end{tabular}
    \caption{Rosenbrock function, init=$4\times\text{rand}(2,1)-2$.}
    \label{tab:rosenbrock}
\end{table}

\begin{figure}
    \centering
    \subfigure[Trajectory of ASK method]{
    \includegraphics[width=0.4\textwidth]{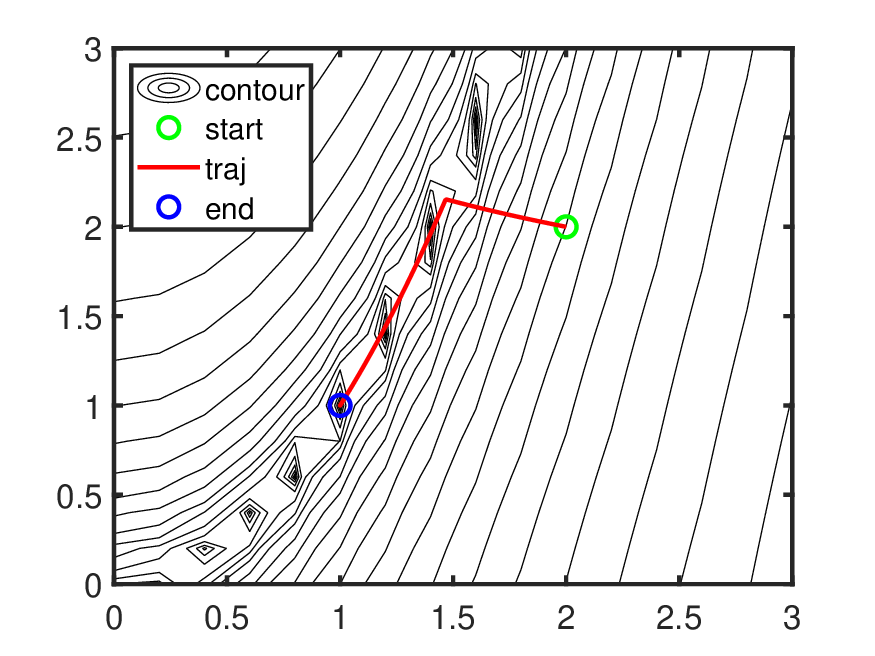}
    \label{fig:RosenTrace}
    }
    \subfigure[$\|\nabla f(\bm x^{(k)}) \|$ v.s.~$k$.]{
    \includegraphics[width=0.4\textwidth]{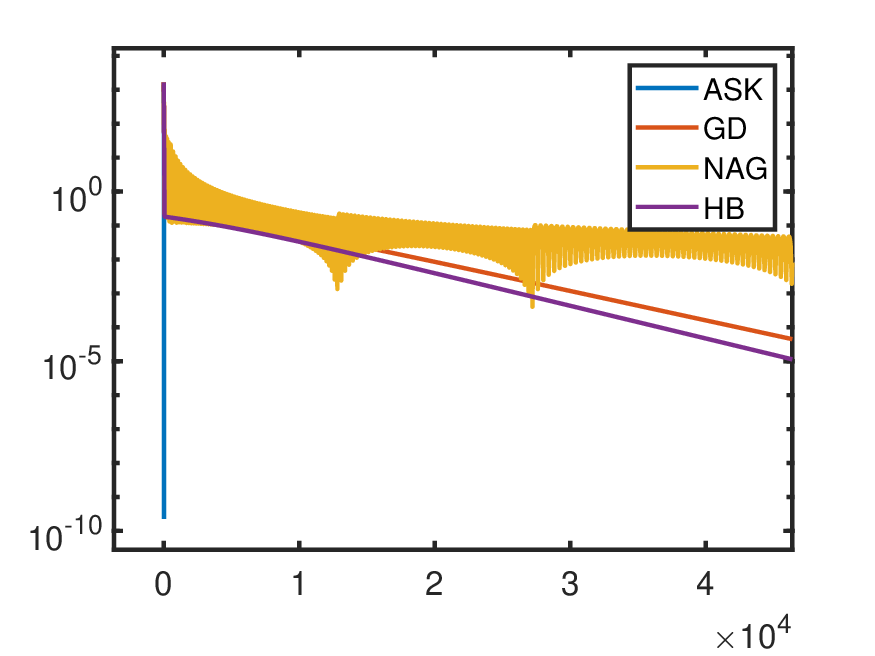}
    \label{fig:RosenGrad}
    }
    \caption{
    Algorithmic behaviors on the Rosenbrock function. All the gradient-based methods (GD, HB, and NAG) encounter difficulties or substantial slowdowns within the asymmetric valley that houses all local minima. In contrast, our proposed algorithm (ASK) manages to traverse efficiently within the valley, pinpointing a local minimum, attributed to its built-in adaptive mechanism.
    }
    \label{fig:rosenbrock}
\end{figure}

\subsubsection{High-dimensional examples}\label{sect:HD}
We investigate two high-dimensional testing functions for minimization. The first testing function is a 10-dimensional Dixon-Price function defined in \cref{eq:Dixonpr} with $d = 10$.
The comparison results are summarized in \cref{tab:Dixonpr10}, indicating that  the 
proposed algorithm returns the gradient norm closest to $0$ with orders of magnitude faster convergence than other methods.

\begin{table}[ht]
    \centering
    \begin{tabular}{ccc}
    \hline\hline
         Methods  &$\|\nabla f\|$& Time \\
         \hline 
         ASK&\textbf{2.2497e-06}&\textbf{6.7327e-02} \\
         GD&3.5192e-04&2.2789e+00 \\
         NAG&1.9299e-03&1.5885e-01 \\
         HB&1.3572e-04&6.1230e-01 \\       
         \hline\hline
    \end{tabular}
    \caption{10-d Dixon-Price Function, init=$20\times\text{rand}(2,1)-10$. }
    \label{tab:Dixonpr10}
\end{table}

The second testing function is the least-squares minimization, which can be expressed as
\begin{align}
    f(\h x) = \frac{1}{2} \h x^{T} A\h x - \h x^{T} \h b,
\end{align}
where $A$ is a  $d\times d$ matrix of full rank and $\h x, \ \h b$ are $d$-dimensional vector satisfying $A\h x = \h b$. 
We vary $d$ by $200,400,600$ and record the corresponding results in  \Cref{tab:lsq}.
In all the three cases, our proposed method finds a solution with its gradient norm several orders of magnitude smaller than the other methods with comparable run time.

\begin{table}[ht]
    \centering
    \begin{tabular}{ccc}
    \hline\hline
         Methods & $\|\nabla f\|$ & Time\\
         \hline
            dim = 200\\
            ASK&	1.8745e-09&	6.8912e+00 \\
            GD&	4.2493e-03&	4.0453e+01 \\
            NAG&	1.4150e-05&	9.6620e+00 \\
            HB&	1.5566e-03&	4.8435e+01 \\
         \hline
            dim = 400\\
            ASK&	5.6230e-11&	9.6911e+01 \\
            GD&	1.5588e-02&	1.4123e+02 \\
            NAG&	2.3108e-05&	4.5453e+01 \\
            HB&	8.7132e-06&	1.3917e+02 \\
         \hline
            dim = 600\\
            ASK&	2.9439e-09&	4.5376e+02 \\
            GD&	1.3730e-03&	2.9747e+02 \\
            NAG&	1.1905e-06&	1.7545e+02 \\
            HB&	5.4395e-05&	3.6776e+02 \\
         \hline\hline
    \end{tabular}
    \caption{Least-squares minimization under with dimensions of $200\times 200, 400\times 400, 600\times 600.$ For each case, ASK achieves the highest accuracy with comparable time. 
    }
    \label{tab:lsq}
\end{table}





\subsection{Min-Max problems}\label{sect:exp-minmax}
In this section, we examine a series of numerical experiments designed to solve the min-max problem \eqref{eq:minmax setup}.
Deliberately selecting four illustrative examples, we aim to conduct an in-depth exploration of the dynamics and efficacy of our proposed methodology:
\begin{enumerate}
    \item[(i)]  A challenging saddle point: we consider a system characterized by a particularly intricate saddle point, a scenario known to pose considerable challenges for traditional gradient-based algorithms. 
    
    \item[(ii)] Unique solution. While simpler in nature, this case underscores the effectiveness of our approach in handling well-defined optima.
    
    \item[(iii)] Multiple solutions: we examine situations with multiple solutions that introduce complexity during the evolution of an algorithm. This case emphasizes the versatility of our proposed method in addressing diverse problem structures.
    
    \item[(iv)] High-Dimensional case. Using a 100-dimensional testing function, we aim to demonstrate the scalability and adaptability of our approach.
\end{enumerate}
Collectively, these examples comprehensively evaluate our proposed methodology, showcasing its performance and adaptability in addressing a spectrum of min-max optimization challenges.

We select a set of well-established gradient-based optimization techniques as benchmarks for comparison, including gradient descent ascent (GDA) \eqref{eq:gradient descend ascend}, optimistic GDA with momentum (OGDA) \eqref{eq:Optimistic Gradient Descent Ascent}, Nesterov's Accelerated Gradient (NAG) \eqref{eq:Nest accel grad}, and the heavy ball (HB) gradient method  \eqref{eq:grad momentum}. 
OGDA, NAG, and HB introduce momentum terms in the optimization process to expedite the convergence of the gradient-based method GDA.  
Due to the complex nature of min-max problems, there are cases when an algorithm fails to converge. Consequently, we report the success rates (i.e., $\|f(\h x)\|<10^{-6}$ along with the gradient norm
and time consumption averaged over all successful trials.



\paragraph{Case (i)}
We start with a simple function defined by
\begin{align}
    f(x_1, x_2) = x_{1}x_{2}.
\end{align}
The Hessian matrix of this system is given by
$H = \begin{bmatrix}
0 & 1 \\
1 & 0
\end{bmatrix}$.
The determinant of the Hessian matrix is $\text{det}(H) = -1$, which implies that $(0, 0)$ is the unique saddle point to the problem:
\[
\min_{x_1}\max_{x_2} x_{1}x_{2}.
\]
Unfortunately,  the saddle point is not stable, as the eigenvalues of this matrix are imaginary. 
Starting from an initial point within the domain $[-1,1]\times[-1,1]$, we present the ASK trajectory in
\Cref{fig:xy_traj} alongside the gradient norm plotted in \Cref{fig:xy_grad}, illustrating plateaus and oscillations in the evolution process. 
As illustrated in \Cref{fig:xy_traj_ogda} and \Cref{fig:xy_grad_ogda}, the trajectory of OGDA circles around $(0,0)$ and the reduction of $\|\nabla f\|$ slows down with each iteration. 
Table \ref{tab:xy result} reports the gradient norm of the solution, the run time, and the success rate of each method over 100 random initial points.
It is evident that ASK is the only method capable of converging to the unique saddle point within a reasonable time frame, whereas both GDA and NAG diverge. The reason ASK succeeds in this example lies in the fact that
the Koopman operator maps a finite-dimensional non-linear dynamical system to an infinitely dimensional linear system, which can be approximated and truncated by ASK. By carefully selecting the acceptable range radius $r$, our truncated approximation of the Koopman operator effectively tracks the descending factor in this problem. In contrast, traditional methods that strictly follow the exact gradient are unable to adapt to the problem's nuances, leading to their slower convergence or in some cases, failure to converge. In summary, ASK's success can be attributed to its adaptability to the underlying structure of the optimization landscape, enhancing its performance, especially in situations where traditional methods encounter difficulties in navigating complex, non-linear dynamics.




\begin{figure}
    \centering
    \subfigure[ASK trajectory]{
        \includegraphics[width=0.4\textwidth]{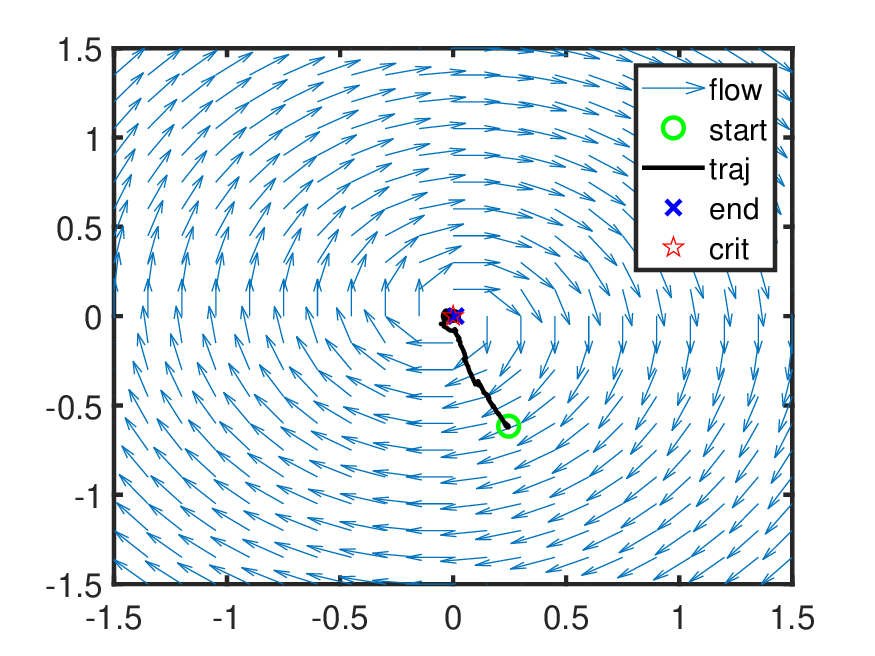}
        \label{fig:xy_traj}
        }
    \subfigure[$\|\nabla f(\bm x^{(k)})) \|$ of ASK]{
        \includegraphics[width=0.4\textwidth]{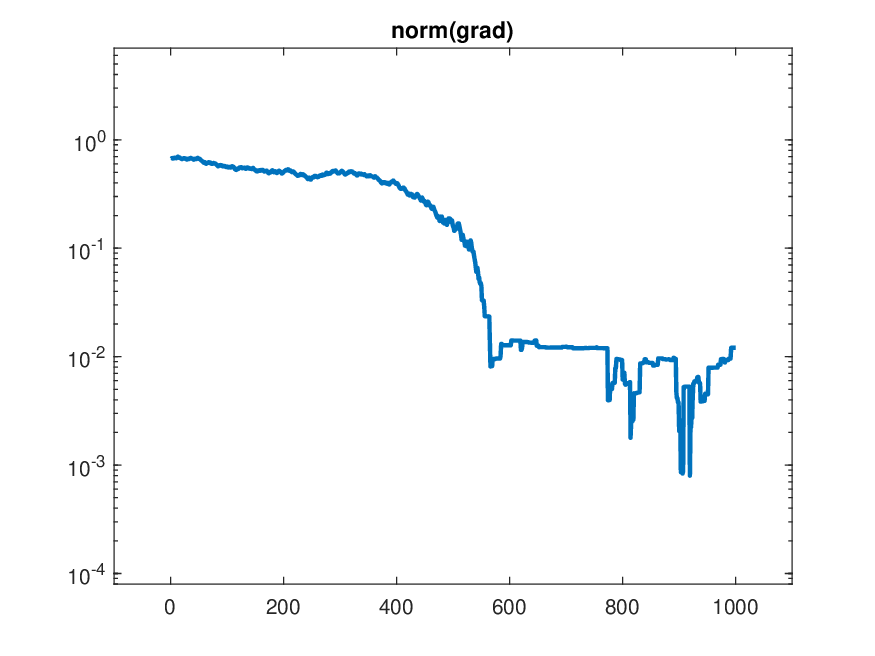}
        \label{fig:xy_grad}  
        }\\
    \subfigure[OGDA trajectory]{
        \includegraphics[width=0.4\textwidth]{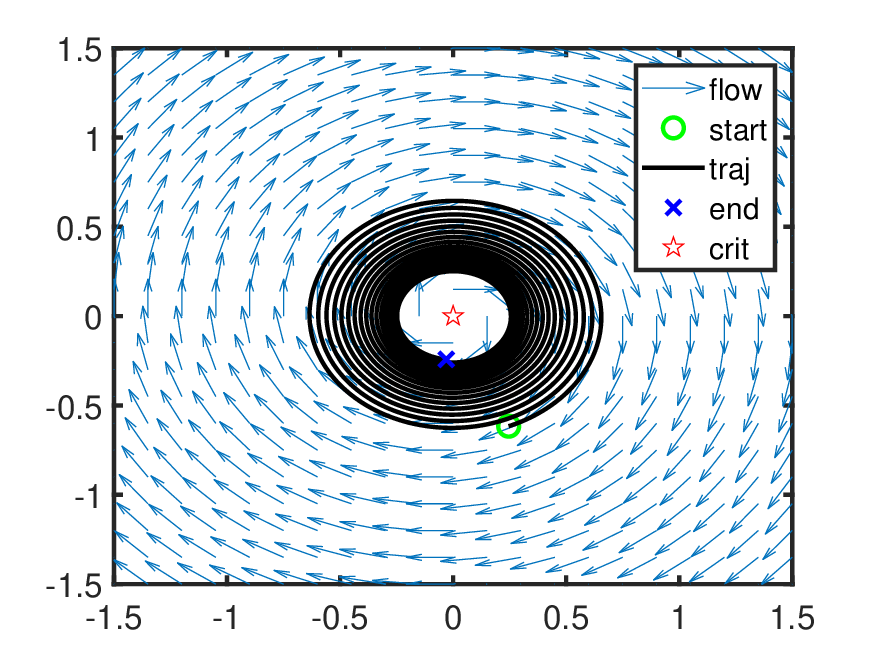}
        \label{fig:xy_traj_ogda}
    }
    \subfigure[$\|\nabla f(\bm x^{(k)})) \|$ of OGDA]{
        \includegraphics[width=0.4\textwidth]{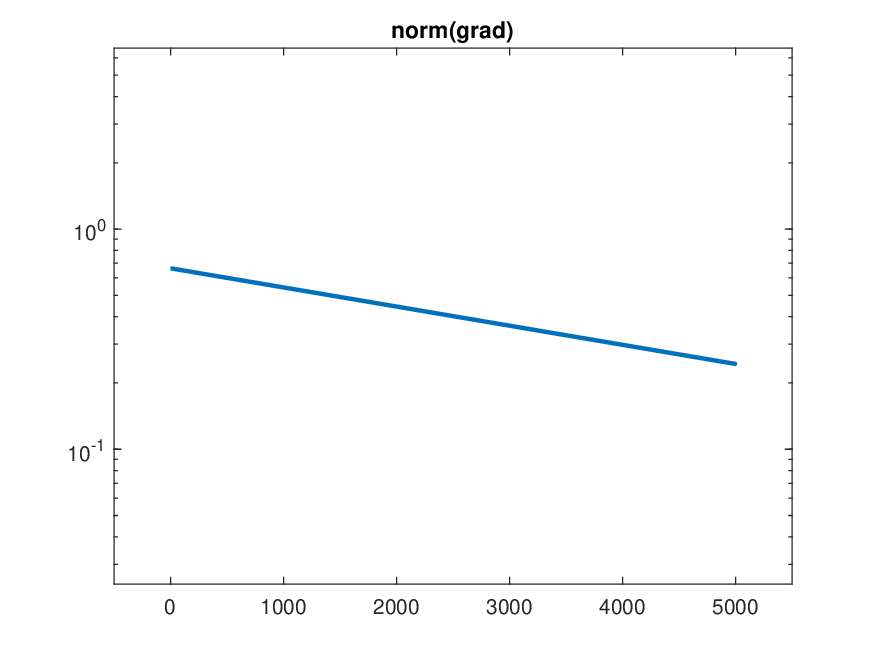}
        \label{fig:xy_grad_ogda}
    }    
    \caption{The trajectory of ASK and OGDA evolution from a particular initial point for solving $\min_{x_1}\max_{x_2} f(x_1,x_2)= x_1x_2.$ OGDA does not work well due to the lack of stability.}
\end{figure}

\begin{table}
    \centering
    \begin{tabular}{cccc}
    \hline\hline
         Methods & $\|\nabla f\|$ & Time(Avg.)& Success rate  \\
         \hline         
         ASK& 5.4423e-03&2.3591e+00& 1.00 \\
         OGDA& 6.5660e-01&6.9367e+01& 0.00  \\
         GDA& 1.3918e+00&9.8197e-03& 0.00  \\
         NAG& 5.2703e+01&6.2550e-03& 0.00  \\
         HB& 1.8026e+00&6.1683e-03& 0.00  \\
         \hline\hline
    \end{tabular}
        \caption{$f(x_1,x_2) = x_{1}x_{2}$, init=$2 \times\text{rand}(2,1)-1.$ The proposed ASK is the only method that can converge to the unique saddle point with 100$\%$ success rates.}
    \label{tab:xy result}
\end{table}



\begin{figure}
    \centering
    \subfigure[ASK trajectory]{
    \includegraphics[width=0.4\textwidth]{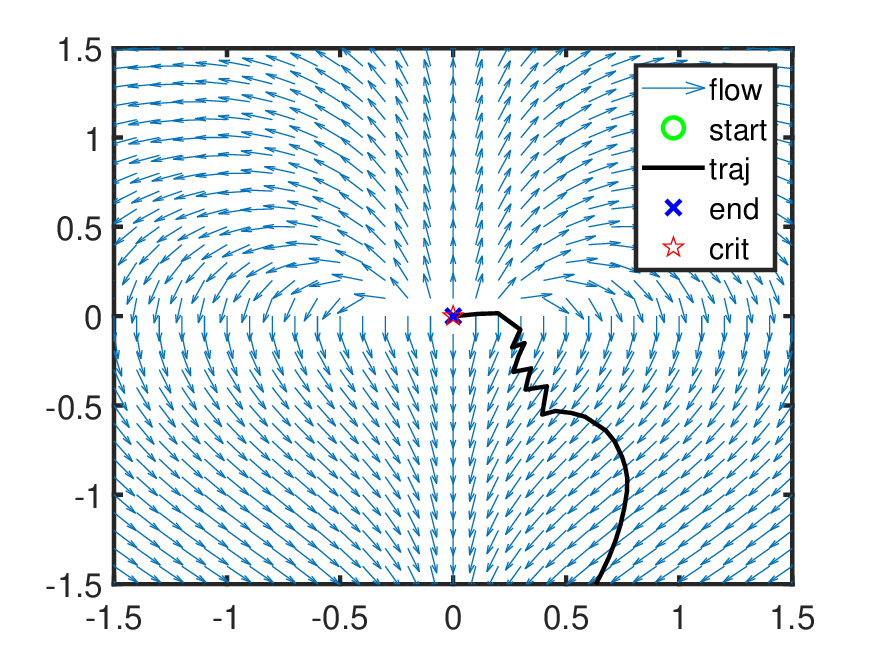}
    \label{fig:simple_traj}}
    \subfigure[$\|\nabla f(\bm x^{(k)})) \|$ of ASK]{
    \includegraphics[width=0.4\textwidth]{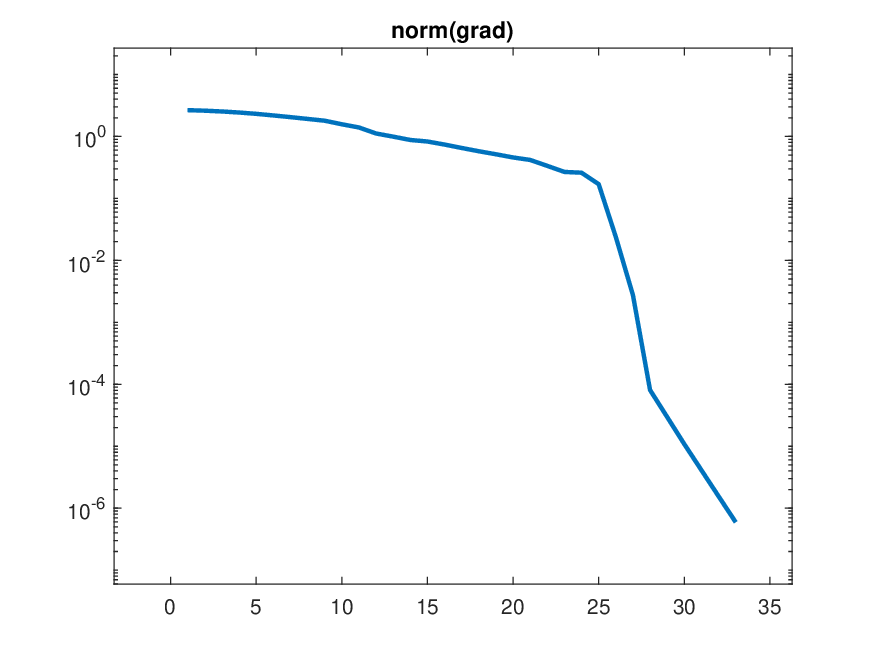}
    \label{fig:simple_grad}}
    \caption{(a) A zig-zag trajectory of the ASK evolution  for solving $\min_{x_1}\max_{x_2} f(x_1,x_2)=-x_{1}^{2}x_{2} + 0.5{x_{2}^{2}}$ from a particular initial point. 
    (b) The evolution of $\|\nabla f\left(\h x^{(n)}\right) \|$ illustrates that the proposed method encounters challenges in the first 25 iterations.  \comm{with some local dynamics}  Here, $x_{1} = 0$ is near the border of the neighborhood and the dynamic along that line is degenerated along the x-axis, which corresponds to a singular matrix. Consequently, the gradient drops rapidly when the trajectory gets close enough to this line.
    }
    \label{fig:simple}
\end{figure}

\begin{table}
    \centering
    \caption{$f(x_1, x_2) = -x_{1}^{2}x_{2} + 0.5{x_{2}^{2}}$, init=$2 \times\text{rand}(2,1) - 1$}
    \begin{tabular}{cccc}
    \hline\hline
         Methods&$\|\nabla f\|$& Time&Success rate\\
         \hline
         ASK& 5.5890e-07&1.3220e-02&1.00 \\
         GDA& 4.5703e-04&9.1963e-02&0.99 \\
         OGDA& 7.6267e-05&2.0828e-01&1.00 \\
         NAG& 1.3524e-3&5.8679e-02&0.51 \\
         HB& 3.9246e-04&5.8969e-02&1.00 \\
         \hline\hline
    \end{tabular}
    \label{tab:simple result}
\end{table}
\paragraph{Case (ii)} We consider the following function
\begin{align}
    f(x_1,x_2) = -x_{1}^{2}x_{2}^{2} + 0.5{x_{2}^{2}},
\end{align}
which has the unique saddle point of $f(\h x^{*}) = 0$ at $\h x^{*} = (0,0)$.
We choose the initial points from $[-1,1]\times[-1,1]$. For a specific initial point, \Cref{fig:simple_traj} shows the trajectory of the ASK evolution, revealing a zig-zag behavior caused by local complexities where the system's dynamics shift.
Consequently, these localized complexities necessitate multiple reevaluations of \comm{the local dynamics via} SVD to closely track an equilibrium.
In \Cref{fig:simple_grad},  the norm of its gradient at each iteration is plotted. 
It appears that the proposed method encounters challenges with \comm{some local dynamics} in the first 25 iterations. Here, $x_{1} = 0$ is near the border of the neighborhood and the dynamic along that line is degenerated along the x-axis, which corresponds to a singular matrix. Consequently, the gradient drops rapidly when the trajectory gets close enough to this line.
 \Cref{tab:simple result} records the gradient norm of the final solution returned by each method, computation time, and success rate over 100 random initial conditions. All the methods, except NAG, demonstrate a high success rate in finding the unique saddle point. ASK achieves the best results in terms of $\|\nabla f(\h x)\|$ and run time. In addition, OGDA achieves higher success rates than GDA and HB but at a cost of approximately twice the time consumption.


\begin{figure}
    \centering
    \subfigure[ASK trajectory]{
    \includegraphics[width=0.4\textwidth]{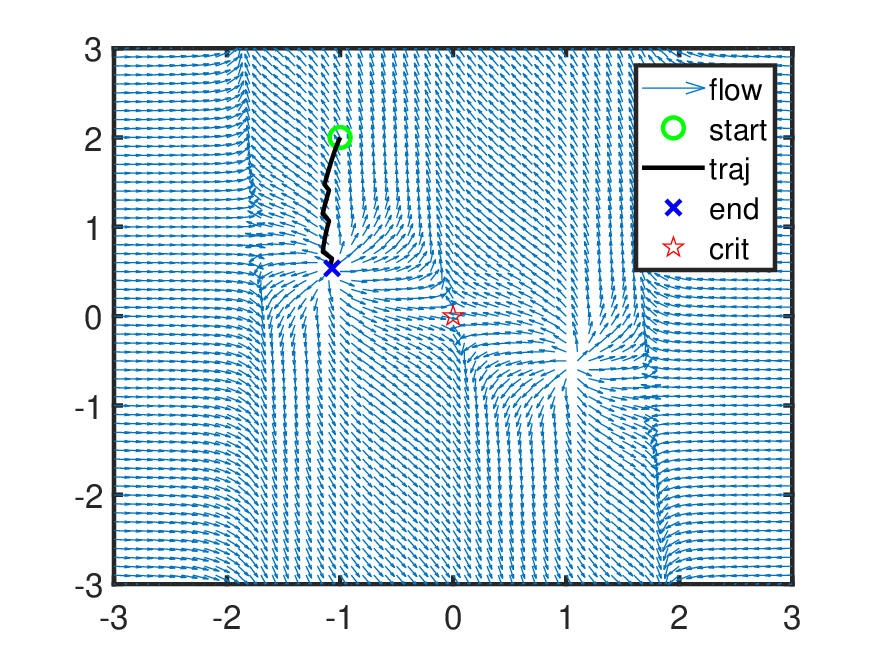}
    \label{fig:camel3MM_traj}}
    \subfigure[$\|\nabla f(\bm x^{(k)})) \|$ of ASK]{
    \includegraphics[width=0.4\textwidth]{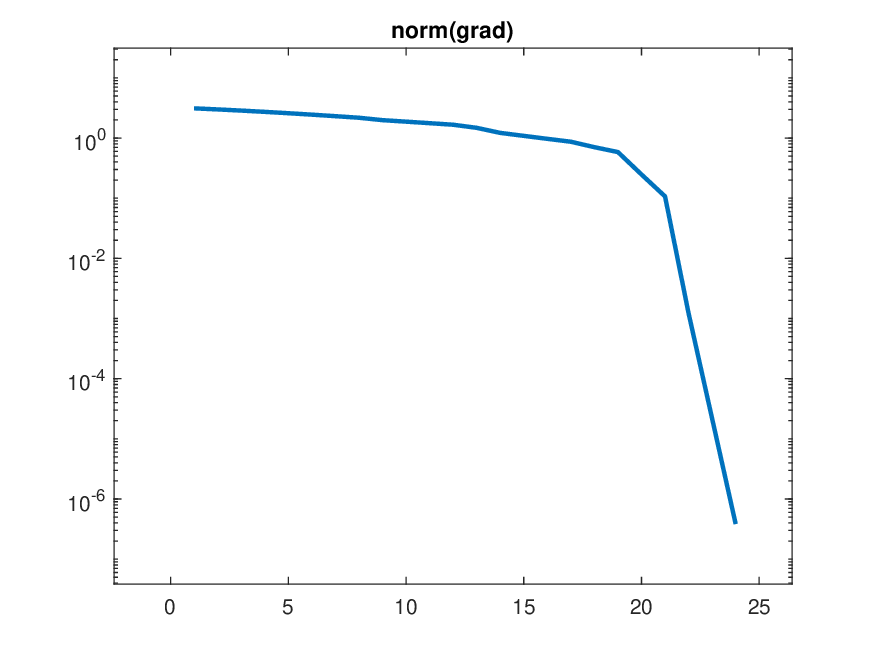}
    \label{fig:camel3MM_grad}}
    \caption{(a) The dynamic of the 3-Hump Camel function and the trajectory of the ASK evolution from a specific initial point. (b) The plot of $\|\nabla f(\bm x^{(k)}))\|$ shows a gradual decrease for the first 20 iterations, followed by a rapid decay to $10^{-6}$. }
    \label{fig:Camel3MM}
\end{figure}

\paragraph{Case (iii)} We investigate the three-hump Camel function, defined in \eqref{eq:camel3}.
\Cref{fig:Camel3MM} presents the ASK evolution and the gradient norm. Our algorithm exhibits a gradual decrease for the first 20 iterations, followed by a rapid decay to $10^{-6}$. The comparisons with other methods are documented in \Cref{tab:Camel3 result}, with each recorded value representing the average over 100 random initial points in the domain of $[-3,3]\times[-3,3]$. The proposed method achieves comparable success rates to GDA, OGDA, and HB but it is an order of magnitude faster. 


\begin{table}
    \centering
    \caption{3-hump Camel function, init = $6 \times\text{rand}(2,1) - 3$}
    \begin{tabular}{cccc}
    \hline\hline
         Methods&$\|\nabla f\|$& Time&Sucess rate\\
         \hline
         ASK& 3.8458e-07&7.8893e-03&0.82 \\
         GDA& 9.9916e-07&4.7122e-02&0.84 \\
         OGDA& 9.9755e-07&4.4569e-02&0.86 \\
         NAG& N/A&N/A&0.00 \\
         HB& 9.9745e-07&3.3747e-02&0.80 \\
         \hline\hline
    \end{tabular}
    \label{tab:Camel3 result}
\end{table}



\paragraph{Case (iv)} The last testing function is the 100-d Dixon-Price function, defined in \eqref{eq:Dixonpr} with $d = 100$.
The comparison results are presented in \Cref{tab:Dixonpr100} averaged over 100 initial points chosen randomly from $[-10,10]^{100}$. 
The proposed algorithm gives the highest success rates while being several orders of magnitude faster than other methods. 

\begin{table}
    \centering
    \caption{100d Dixon-Price, init = $20 \times\text{rand}(100,1) - 10$}
    \begin{tabular}{cccc}
    \hline\hline
         Methods&$\|\nabla f\|$& Time&Sucess rate\\
         \hline
         ASK& 6.5132e-07&7.8893e-03&0.85 \\
         GDA& N/A&N/A&0.00 \\
         OGDA& N/A&N/A&0.00 \\
         NAG& N/A&N/A&0.00 \\
         HB& 5.9745e-02&4.1235e+01&0.20 \\
         \hline\hline
    \end{tabular}
    \label{tab:Dixonpr100}
\end{table}

\section{Conclusions}
\label{sec:dicussion_conclusion}

In this paper, we have harnessed the power of the Adaptive Spectral Koopman (ASK) method \cite{li2023adaptive} with sparse grids \cite{li2022sparse} that was originally designed for dynamical systems to general optimization problems, including a min-max problem. 
Instead of closely tracking a single trajectory, we adapted the ASK method in such a way that it focuses on the dynamics of the system and identifies its   equilibrium, which corresponds to a (local) critical point of the relevant optimization problem.  We provided an error estimate of a special case, which connects the accuracy of the spectral methods based on sparse grid points and the proposed method. 
We conducted extensive experiments on various testing cases for a thorough assessment of ASK's capabilities.
Our numerical findings consistently highlight the advantage of accuracy obtained by the ASK method, indicating its superior precision in finding critical points over some popular gradient-based methods. Remarkably, this elevated accuracy does not incur a substantial computational cost, as ASK's computational time remains comparable to, or marginally higher than, that of traditional optimization methods. 
The experimental results on min-max problems suggested an improved capability of ASK to identify critical points, particularly in scenarios where \comm{local dynamics} system's dynamics in the neighborhood at the current state effectively approximate the system behaviors globally.
Collectively, this research demonstrates the efficacy and versatility of the ASK method as a powerful tool for a wide spectrum of optimization problems. Its remarkable accuracy, coupled with reasonable computational costs, positions it as a promising candidate for real-world applications. 

\bibliographystyle{siam}
\bibliography{references}
\end{document}

%% file: analysis.tex
\subsection{Numerical Analysis}\label{sect:anal}

In general, conducting a comprehensive numerical analysis of the ASK-based optimization
algorithm is a complex task. The error in the ASK expansion for a dynamical system
includes the following  three components:
\begin{enumerate}
\item Errors in discretizing the Koopman generator $\mathcal{K}$.
Despite being a linear operator, it is noteworthy that the eigenvalues of $\mathcal K$ 
only constitute a discrete spectrum. Consequently, when we employ a matrix $\mathbf{K}$ to approximate $\mathcal{K}$, 
we may lose information in the continuous spectrum of $\mathcal{K}$ beyond the errors introduced in the approximation of eigenvalues.
\item Errors in approximating global functions locally. 
The eigenfunction $\varphi_j$ is defined on the entire state space, but in computation, we approximate it within a bounded domain. Consequently, employing local information to approximate global functions, such as eigenfunctions and Koopman modes, can result in inaccuracies. This is the rationale behind the introduction of the adaptivity step in the design of ASK. Through this step, we aim at a more accurate approximation of global functions $\varphi_j$ across various bounded domains within the state space.
\item Errors in approximating eigenfunctions $\varphi_j$ using polynomials. We
employ the standard spectral method to approximate eigenfunctions $\varphi_j$ using polynomials
$\varphi_j^N$ within a bounded domain, a practice that may result in errors in the approximation.
Similarly, alternative approximation approaches, such as radial basis function 
approximation, Fourier series approximation, and wavelet approximation, can 
introduce various numerical errors. 
\end{enumerate}
In this work, we focus on addressing the third type of error mentioned earlier, specifically investigating the approximation in the vicinity of a critical point by using our algorithm. 
 Our objective is to leverage numerical analysis tools in both the spectral method and the sparse grids method to provide error estimates, which allows us to gain a preliminary understanding of the convergence of the proposed method, as outlined in \Cref{thm:ask_opt}.

\begin{theorem}
  \label{thm:ask_opt}
Consider a  (truncated) Koopman expansion of a scalar function $g\in\mathbb R^d\rightarrow \mathbb R$, i.e.,
\[
g(\bm x(t_0+t)) = \sum_{j=0}^{N} c_{j} \varphi_{j}(\h x(t_{0})) \exp(\lambda_{j}t),
\]
 and its ASK approximation $\sum_{j=0}^{N} \tilde{c}_{j} \varphi^{N}_{j}(\h
    x(t_{0})) \exp (\tilde\lambda_{ j } t )$ in a bounded neighborhood of $\bm
    x(t_0)$, denoted as $B_{\bm x(t_0)}(r):=\{\bm y\in \mathbb{R}^d
    | \Vert \bm x(t_0)-\bm y\Vert_{\infty} \leq r\}$,
  where eigenpairs $(\varphi_j, \lambda_j)$ are approximated by
    $(\varphi_j^N,\tilde\lambda_j)$.
   Define $\varepsilon_{\varphi}=\max_{0\leq j\leq N} \{\Vert \varphi_j-\varphi_j^N\Vert_{B_{\bm
    x(t_0)}(r),\infty}\}$,
    $\varepsilon_{\lambda}=\max_{0\leq j\leq N}\{\vert\lambda_j-\tilde{\lambda}_j\vert \}$, and the matrix $\tensor\Phi$ with each element $\Phi_{ij}=\varphi_j(\bm \xi_i)$ with its condition number denoted by $\kappa.$
  If $\lambda_j$ and $\tilde \lambda_j$ have nonpositive real part, 
    $\operatorname{Re}(\tilde\lambda_j)<0$ for $\operatorname{Re}(\lambda_j)<0$, and $\kappa(N+1) \varepsilon_{\varphi}<\Vert\tensor\Phi\Vert_{\infty}$,    we 
    have: 
    \begin{equation}
      \label{eq:ask_bound}
    \left\Vert \sum_{j=0}^{N} c_{j} \varphi_{j}(\bm x(t_{0})) \exp(\lambda_{j}t)
    - \sum_{j=0}^{N} \tilde{c}_{j} \varphi^{N}_{j}(\bm x(t_{0})) \exp
      (\tilde\lambda_{ j } t ) \right\Vert \lesssim C_1 \varepsilon_{\lambda} t +
      C_2\varepsilon_{\varphi},
    \end{equation}
    where both $C_1$ and $C_2$ depends on $\Vert\bm c\Vert_{\infty}$ and $\max_j
    |\varphi_j(x(t_0))|$ with $C_2$ additionally dependent on condition number of $\tensor\Phi$ and $N$.
\end{theorem}
\begin{proof}
 The Koopman modes $c_j$  satisfy $\tensor \Phi\bm c=\bm g$ with $g_i=g(\bm \xi_i)$. Since $\varphi_j$ are approximated by
  $\varphi_j^N$, the perturbation of $\bm\Phi$, defined as $\delta\tensor\Phi=\tensor\Phi-\tensor\Phi^N$, satisfies $\Vert\delta\bm\Phi
  \Vert_{\infty}\leq (N+1) \varepsilon_{\varphi}$.  Recall that $\tensor\Phi^N\tilde{\bm c}=\bm g$, and hence
we get
\begin{equation}
    \dfrac{\Vert \bm c-\tilde{\bm c} \Vert_{\infty}}{\Vert\bm c
  \Vert_{\infty}}\leq\dfrac{\kappa\Vert\delta\tensor\Phi \Vert_{\infty} }{\Vert\tensor\Phi\Vert_{\infty}
  -\kappa\Vert\delta\tensor\Phi \Vert_{\infty} } \leq
  \dfrac{\kappa (N+1)\varepsilon_{\varphi}}{\Vert \tensor \Phi\Vert_{\infty} -
  \kappa (N+1)\varepsilon_{\varphi}}:=\varepsilon_{\infty},
\end{equation}
where we use the assumption that   $\kappa(N+1)\varepsilon_{\varphi} < \Vert\bm\Phi \Vert_{\infty}$.
 If
  $\varepsilon_{\varphi}$ is sufficiently small  such that $\kappa(N+1)\varepsilon_{\varphi}\ll \Vert \tensor\Phi\Vert_{\infty}$, then $\varepsilon_{\infty}\approx
  \frac{\kappa(N+1)}{\Vert \tensor\Phi \Vert_{\infty}}\varepsilon_{\varphi}$. Define a set $S=\{j\in \{0, 1, \cdots,
  N\}|\operatorname{Re}(\lambda_j)<0\}$ and $S^c = \{0, 1,\cdots, N\}\backslash S$.  By the assumption that 
  $\operatorname{Re}(\lambda_j)$ and $\operatorname{Re}(\tilde\lambda_j)$ are
  nonpositive, we have
  \begin{equation}
    \label{eq:estimate_part1}
   \begin{aligned}
    & \left\Vert \sum_{j\in S} c_{j} \varphi_{j}(\bm x(t_{0})) \exp(\lambda_{j}t)
    - \sum_{j\in S} \tilde{c}_{j} \varphi^{N}_{j}(\bm x(t_{0})) \exp
      (\tilde\lambda_{ j } t ) \right\Vert \\
      \leq & (N+1)\left(
      \Vert \bm c\Vert_{\infty}\exp(\lambda_{\max} t) +  
      \Vert \tilde{\bm c}\Vert_{\infty}\exp(\tilde\lambda_{\max} t)\right) \\
      \leq & (N+1)\Vert\bm c\Vert_{\infty}\left( \exp(\lambda_{\max} t) +
      (1+\varepsilon_{\infty}) \exp(\tilde\lambda_{\max} t)\right),
  \end{aligned}
  \end{equation}
 where $\lambda_{\max}=\max_{j\in S} \{\operatorname{Re}(\lambda_j)\}<0$ and
  $\tilde\lambda_{\max}=\max_{j\in S} \{\operatorname{Re}(\tilde\lambda_j)\}<0$.
  Therefore, the magnitude of the difference in \eqref{eq:estimate_part1} converges to zero exponentially as $t\rightarrow\infty$.

  By the assumption that $\operatorname{Re}(\tilde\lambda_j)<0$ for $\operatorname{Re}(\lambda_j)<0$, we have 
  $\operatorname{Re}{\tilde\lambda_j}\leq \operatorname{Re}{\lambda_j}=0$ for $j\in S^c$, which implies that
    \[ \begin{aligned}
    & \left\vert c_{j} \varphi_{j}(\bm x(t_{0})) \exp(\lambda_{j}t)
    - \tilde{c}_{j} \varphi^{N}_{j}(\bm x(t_{0})) \exp
      (\tilde\lambda_{ j } t ) \right\vert \\
      \leq & \big\vert c_{j} \varphi_{j}(\bm x(t_{0})) \exp(\lambda_{j}t)
    - c_{j} \varphi_{j}(\bm x(t_{0})) \exp(\lambda_{j}t)\exp ((\tilde\lambda_{ j
    }-\lambda_j) t ) \\
    &+ 
    (c_{j} \varphi_{j}(\bm x(t_{0}))- \tilde{c}_{j} \varphi^{N}_{j}(\bm
    x(t_{0}))) \exp (\tilde\lambda_{ j } t ) \big\vert \\
    \leq & \left\vert c_{j} \varphi_{j}(\bm x(t_{0}))
    \exp(\lambda_{j}t)\right\vert\cdot\left\vert 1- \exp ((\tilde\lambda_{ j
    }-\lambda_j) t ) \right\vert + 
    \vert c_j\varphi_j(t_0) - c_j\varphi_j^N(t_0) \vert  + 
    \vert c_j\varphi_j^N(t_0) -\tilde c_j\varphi_j^N(t_0) \vert   \\
    \leq & \left\vert c_{j} \varphi_{j}(\bm x(t_{0})) \right\vert\cdot
    \left\vert 1- \exp ((\tilde\lambda_{ j }-\lambda_j) t ) \right\vert  + 
    \Vert \bm c\Vert_{\infty}\varepsilon_{\varphi}  + 
    \Vert\bm c \Vert_{\infty}\varepsilon_{\infty} (\max\{\varphi_j(t_0)\} +
    \varepsilon_{\varphi})   .
    \end{aligned}\]
 Therefore, when $\vert \tilde{\lambda}_j-\lambda_j \vert$ is sufficiently small, we have
  $\left\vert 1- \exp ((\tilde\lambda_j-\lambda_j) t ) \right\vert=
  \mathcal{O}(\varepsilon_{\lambda}t)$.
  In practice, 
  for a given small number $\epsilon$, we
identify  $t$ such that the difference in \cref{eq:estimate_part1} is upper bounded by
  $\epsilon$, thereby establishing the validity of the desired bound in \eqref{eq:ask_bound}.
\end{proof}

  Of note, 
  in~\cref{thm:ask_opt}, it is a requirement that
  $(\varphi_j^N,\tilde\lambda_j)$ serves as an accurate approximation of $(\varphi_j, \lambda_j)$. Utilizing the spectral method, we generally observe that when eigenfunctions $\varphi_j$  exhibit good regularity---meaning high-order (mixed) derivatives exist and are bounded---we can anticipate a high level of accuracy in the approximation, as shown in the study of eigenvalue problems by using the spectral method. However, the majority of such works are associated with self-adjoint operators under specific boundary conditions, a setup that differs from the one in ASK. Therefore, a comprehensive error estimate requires careful investigation and will be addressed in future work. In the ideal case, we may expect
  $\varepsilon_{\lambda}$ and $\varepsilon_{\varphi}$ to be at the order of
  $\mathcal{O}(n^{-k})$ with the accuracy of interpolation based on sparse
  grid~\cite{barthelmann2000high}, where $n$ is the total number of sparse grids
  points and $\varphi_j\in C^k(B_{\bm x(t_0)}(r))$. 